\documentclass[11pt,a4paper,leqno,noamsfonts]{amsart}
   \makeatletter 
   \renewcommand\@biblabel[1]{#1.}
   \makeatother
\usepackage[english]{babel}
\usepackage[dvipsnames]{xcolor}
\usepackage{graphicx,pifont,soul}
\usepackage[all,cmtip]{xy}
\usepackage{todonotes}
\usepackage[utopia]{mathdesign}
\usepackage[a4paper,top=2.8cm,bottom=2.8cm,
    left=2.8cm,right=2.8cm,marginparwidth=60pt]{geometry}
\usepackage[utf8]{inputenc}
\usepackage{nicefrac}
\usepackage{braket,caption,comment,mathtools,stmaryrd}
\usepackage[usestackEOL]{stackengine}
\usepackage{multirow,booktabs,microtype,relsize}
\usepackage{tikz,tikz-cd}
    \numberwithin{equation}{section}
\usepackage[colorlinks,bookmarks]{hyperref} %
  \hypersetup{colorlinks,%
              citecolor=britishracinggreen,%
              filecolor=black,%
              linkcolor=cobalt,%
              urlcolor=cornellred}
\usepackage[capitalise]{cleveref}

\usepackage{mathbbol}

\numberwithin{equation}{section}
\DeclareSymbolFont{usualmathcal}{OMS}{cmsy}{m}{n}
\DeclareSymbolFontAlphabet{\mathcal}{usualmathcal}
\DeclareMathAlphabet\BCal{OMS}{cmsy}{b}{n}

\definecolor{cornellred}{rgb}{0.7, 0.11, 0.11}
\definecolor{britishracinggreen}{rgb}{0.0, 0.26, 0.15}
\definecolor{cobalt}{rgb}{0.0, 0.28, 0.67}
\usepackage{comment,braket}
\usepackage[shortlabels]{enumitem}

\theoremstyle{definition}

\newtheorem*{lemma*}{Lemma}
\newtheorem*{theorem*}{Theorem}
\newtheorem*{example*}{Example}
\newtheorem*{fact*}{Fact}
\newtheorem*{notation*}{Notation}
\newtheorem*{definition*}{Definition}
\newtheorem*{prop*}{Proposition}
\newtheorem*{remark*}{Remark}
\newtheorem*{corollary*}{Corollary}

\newtheorem*{conventions*}{Conventions}

\newtheorem{definition}{Definition}[section]

\newtheorem{example}[definition]{Example}

\newtheorem{remark}[definition]{Remark}

\newtheorem*{claim*}{Claim}

\newtheoremstyle{thm} % <name> % (ambienti con dimostrazione)
        {3mm}% <Space above>
        {3mm}% <Space below>
        {\slshape}% <Body font> % 
        {0mm}% <Indent amount>
        {\bfseries}% <Theorem head font>
        {.}% <Punctuation after theorem head>
        {1mm}% <Space after theorem head>
        {}% <Theorem head spec (can be left empty, meaning 'normal')> 
\theoremstyle{thm}
\newtheorem{theorem}[definition]{Theorem}
\newtheorem{corollary}[definition]{Corollary}
\newtheorem{lemma}[definition]{Lemma}
\newtheorem{proposition}[definition]{Proposition}
\newtheorem{thm}{Theorem}

\newcommand{\claudio}[1]{\todo[inline, size=\tiny, author=Claudio, backgroundcolor=red!30]{#1}}
\newcommand{\francesco}[1]{\todo[inline, size=\tiny, author=Francesco, backgroundcolor=green!30]{#1}}

\newcommand{\C}{\mathbb{C}}

\newcommand{\Z}{\mathbb{Z}}

\newcommand{\cC}{\mathcal{C}}
\newcommand{\cE}{\mathcal{E}}
\newcommand{\cF}{\mathcal{F}}

\newcommand{\cH}{\mathcal{H}}

\newcommand{\cO}{\mathcal{O}}
\newcommand{\cX}{\mathcal{X}}

\newcommand{\KK}{\mathbb{k}}

\newcommand{\mU}{\mathcal{U}}

\newcommand{\debar}{\bar{\partial}}

\newcommand{\Art}{\mathsf{Art}}

\newcommand{\Sets}{\mathsf{Set}}

\DeclareMathOperator{\Db}{{D^{\mathrm{b}}}}
\DeclareMathOperator{\Aut}{Aut}

\DeclareMathOperator{\Br}{Br}

\DeclareMathOperator{\oH}{H}

\DeclareMathOperator{\Spec}{Spec}
\DeclareMathOperator{\Ext}{Ext}
\DeclareMathOperator{\tr}{tr}
\DeclareMathOperator{\End}{End}

\DeclareMathOperator{\Hom}{Hom}
\DeclareMathOperator{\RHom}{RHom}

\DeclareMathOperator{\Def}{Def}
\DeclareMathOperator{\SU}{SU}

\newcommand*{\defeq}{\mathrel{\vcenter{\baselineskip0.5ex \lineskiplimit0pt
\hbox{\scriptsize.}\hbox{\scriptsize.}}}%
=}

\DeclareMathOperator{\id}{id}

\author{Francesco Meazzini}
\address{\newline
Universit\`a di Roma Sapienza\hfill\newline
Dipartimento di Matematica G. Castelnuovo\hfill\newline
Piazzale Aldo Moro 5, 00185 Roma, Italia}
\email[]{francesco.meazzini@uniroma1.it}

\author{Claudio Onorati}
\address{\newline
Alma Mater studiorum Universit\`a di Bologna\hfill\newline
Dipartimento di Matematica\hfill\newline
Piazza di porta San Donato 5, 40126 Bologna, Italia}
\email[]{claudio.onorati@unibo.it}

\title[Deformations of twisted sheaves and formality problems]{Deformations of twisted sheaves and formality results}

\begin{document}

\begin{abstract}
We show that infinitesimal deformations of twisted sheaves are controlled by the DG Lie algebra of their derived automorphisms. We prove that such DG Lie algebra is formal for polystable twisted sheaves on minimal surfaces of Kodaira dimension $0$ and for projectively hyper-holomorphic locally free twisted sheaves on hyper-K\"ahler manifolds.
\end{abstract}

\maketitle

\tableofcontents

\thispagestyle{empty}

%%%%%%%%%%%%%%%%%%%%%%%%%%%%%%%%%%%%%%%%%%%%%%%%%%%%%%%%%%%%%%
%%%%%%%%%%%%%%%%%%%%%%%%%%%%%%%%%%%%%%%%%%%%%%%%%%%%%%%%%%%%%%
%%%%%%%%%%%%%%%%%%%%%%%%%%%%%%%
%%%%%%%%%%%%%%%%%%%%%%%%%%%%%%%
%%%%%%%%%%%%%%%%%%%%%%%%%%%%%%%

\section*{Introduction}

The first purpose of this paper is to give a rigorous proof of the following folklore result.

\begin{thm}[Theorem~\ref{thm:main}]
    Let $X$ be a smooth projective variety over a field $\KK$ of characteristic $0$, or a projective complex manifold, and let $\alpha\in\Br(X)$ a Brauer class. The infinitesimal deformations of a coherent $\alpha$-twisted sheaf $F$ on $X$ are controlled by the quasi-isomorphism class of DG Lie algebras $\RHom_{(X,\alpha)}(F,F)$.
\end{thm}

When the $\alpha$-twisted sheaf $F$ is locally free, we can drop the projectivity assumptions (see Theorem~\ref{thm:locally free}). 

We refer to \cite{PhD:Caldararu,HuybrechtsStellari,HuybrechtsStellari:ProofofCaldararu} for a general treatment of twisted sheaves and to \cite{Lieblich:TwistedSheaves,Yoshioka:Twisted} for a description of first order deformations and moduli spaces of twisted sheaves.

Let us briefly recall what the statement above means. A differential graded (DG) Lie algebra is a DG vector space $(L=\bigoplus_{i\in\Z}L^i,d)$ endowed with a graded bracket $[-,-]\colon L\times L\to L$ satisfying certain properties. 
To a DG Lie algebra $L$ one can associate a deformation functor $\Def_L\colon\Art_\KK\to\Sets$. Then one says that a given deformation problem $\mathcal{D}\colon\Art_\KK\to\Sets$ is \emph{controlled} by the DG Lie algebra $L$ if there exists a natural isomorphism $\mathcal{D}\cong\Def_L$ of functors. One advantage of this approach is that the DG Lie algebra $L$ provides both the tangent space $\oH^1(L)$ and a complete obstruction theory with values in $\oH^2(L)$ for the controlled deformation problem. We refer to the book \cite{Manetti:LMDT} for details concerning the approach to deformation theory via DG Lie algebras.

When the Brauer class is trivial, the statement is proved in \cite{FIM} and our arguments closely follow theirs.

The idea of the proof of Theorem~A is classical; for simplicity, let us suppose that $F$ is locally free. Then up to take a sufficiently refined open cover of $X$, the deformation data are reduced to gluing data, which are then encoded in the claimed DG Lie algebra using the formalism of semicosimplicial DG Lie algebras (\cite{FMM:Semicosimplicial}).
\medskip

As a corollary, we extend to twisted sheaves formality results of \cite{BMM:Formality,Budur_Zhang_2019,MO22}.

Recall that a DG Lie algebra is \emph{formal} if it is quasi-isomorphic to its cohomology. Geometrically, this implies the quadraticity of the corresponding Kuranishi map. More precisely, in this case we have that the versal deformation can be described via quadratic equations in $\oH^1(L)$ as the preimage of the obstruction map. We also wish to point out that, for a coherent sheaf on a smooth projective surface, formality and quadraticity are indeed equivalent~\cite{BMM:quadraticity}. %In general, a DG Lie algebra $L$ provides both the tangent space $\oH^1(L)$ and a complete obstruction theory with values in $\oH^2(L)$ for the controlling deformation problem; the formality of $L$ then implies that the base of the versal deformation can be described via quadratic equations in $\oH^1(L)$ as the preimage of the obstruction map.

\begin{thm}[Theorem~\ref{thm:kodaira 0} and Theorem~\ref{thm:formality of twisted hyper-holomorphic}]\text{}
    \begin{enumerate}
        \item Let $S$ be a smooth and projective minimal surface of Kodaira dimension $0$ and $H$ an ample class. If $F$ is a $H$-polystable $\alpha$-twisted sheaf, then the DG Lie algebra $\RHom_{(S,\alpha)}(F,F)$ is formal.
        \item Let $X$ be a compact hyper-K\"ahler manifold and $\omega$ a K\"ahler class. If $E$ is a $\omega$-projectively hyper-holomorphic locally free $\alpha$-twisted sheaf, then the DG Lie algebra $\RHom_{(X,\alpha)}(E,E)$ is formal.
    \end{enumerate}
\end{thm}

Here a locally free $\alpha$-twisted sheaf on a compact hyper-K\"ahler manifold is called $\omega$-\emph{projectively hyper-holomorphic} if the untwisted endomorphism sheaf $\cE nd(E)$ is $\omega$-hyper-holomorphic, i.e.\ it deforms along the twistor line associated to $\omega$ (\cite{Verbitsky:hyperholomorphic}). Projectively hyper-holomorphic twisted sheaves are already present in literature (e.g.\ \cite{Markman:modular}) and they have good modular properties. For example, they have been used to prove the D-equivalence conjecture (\cite{MSYZ:D-Equiv}) and to construct new projective models of hyper-K\"ahler manifolds (\cite{Bottini:OG10twisted}).
We point out though that it is still unclear what the right notion of a twisted hyper-holomorphic locally free sheaf is. We speculate on it in Remark~\ref{rmk:speculation}.

%%%%%%%%%%%%%
\subsection*{Plan of the paper}
We collect in Section~\ref{section:twisted} the basic properties of twisted sheaves. In particular, we define the deformation functor of a twisted sheaf and prove that it does not depend on the choice of the representative of the Brauer class and the open cover. In Section~\ref{section:algebra} we collect the algebraic background needed for the proof of Theorem~A, which is then proved in Section~\ref{section:DG Lie of twisted}. We divide the proof in two cases, the locally free case and the general case, each proved in a different section. Beside making the arguments more clear, this has the advantage of make evident where the hypothesis are used. Finally, we prove Theorem~B in Section~\ref{section:surfaces} and Section~\ref{section:HK}, respectively addressing the case of minimal surfaces of Kodaira dimension $0$ and the case of hyper-K\"ahler manifolds.

%%%%%%%%%%%%%%
\subsection*{Acknowledgements}
We wish to thank Donatella Iacono, Emanuele Macrì, Marco Manetti and Paolo Stellari for interest and important discussions at several stages of this work.

The first author is partially supported by the PRIN 20228JRCYB “Moduli spaces and special varieties”, of “Piano Nazionale di Ripresa e Resilienza, Next Generation EU”.

The second author is partially funded by the European Union - NextGenerationEU under the National Recovery and Resilience Plan (PNRR) - Mission 4 Education and research - Component 2 From research to business - Investment 1.1 Notice Prin 2022 - DD N. 104 del 2/2/2022, from title "Symplectic varieties: their interplay with Fano manifolds and derived categories", proposal code 2022PEKYBJ – CUP J53D23003840006.

Both authors are members of INDAM-GNSAGA.

%%%%%%%%%%%%%%%%%%%%%%%%%%%%%%%%%%%%%%%%%%
%%%%%%%%%%%%%%%%%%%%%%%%%%%%%%%%%%%%%%%%%%
\section{Twisted sheaves}\label{section:twisted}

%%%%%%%%%%%%%%%%%%%%%%%%%%%%%%%%%%%%%%%%%%
\subsection{Basics on twisted sheaves}

We consider a field $\KK$ of characteristic $0$, and we work on a finite-dimensional, separated, quasi-compact, quasi-projective, Noetherian $\KK$-scheme $X$. We will also consider the analytic case, in which case $X$ will be a complex manifold. In this section we mostly focus on the algebraic setting, the analytic one being similar.

Recall that the \emph{cohomological Brauer group} $\Br'(X)$ is by definition the torsion part of the cohomology group $\oH^2_{\text{\'et}}(X,\cO_X^*)$. On the other hand, the \emph{geometric Brauer group} $\Br(X)$ is the group of equivalence classes of Azumaya $\cO_X$-algebras. We will always work with the geometric Brauer group $\Br(X)$; recall though that $\Br'(X)=\Br(X)$ for quasi-projective varieties, see~\cite[Theorem~1.1]{deJong:Gabber}.
\medskip

Let then $\alpha\in\Br(X)$ be a Brauer class, and let us choose a \v{C}ech representative $\tilde{\alpha}=\{\alpha_{ijk}\}$ for it with respect to an open covering $\mU=\{U_i\}_{i\in I}$.
\begin{definition}\label{defn:twisted sheaf}
    An $(\tilde{\alpha},\mU)$-\emph{twisted sheaf} $F$ on $X$ is the data of a collection of sheaves $\{F_i\}_{i\in I}$ of $\cO_{U_i}$-modules and a collection of isomorphisms $\{g_{ij}\colon F_i|_{U_{ij}}\to F_j|_{U_{ij}}\}_{i,j\in I}$ such that:
    \begin{enumerate}
        \item $g_{ii}=\id$;
        \item $g_{ij}^{-1}=g_{ji}$;
        \item $g_{ij}\circ g_{jk}\circ g_{ki}=\alpha_{ijk}\id$.
    \end{enumerate}
\end{definition}
A \emph{morphism} $f\colon F\to G$ of $(\tilde{\alpha},\mU)$-twisted sheaves is given by a collection $\{f_i\colon F_i\to G_i\}$ of morphisms that commute with the local isomorphisms. 

An $(\tilde{\alpha},\mU)$-twisted sheaf $F$ is \emph{coherent} if $F_i$ is a coherent sheaf for every $i\in I$.

In this way one can define an abelian category $\operatorname{Coh}(X,\tilde{\alpha},\mU)$ of coherent $(\tilde{\alpha},\mU)$-twisted sheaves.

\begin{lemma}[\protect{\cite[Lemma~1.2.3, Lemma~1.2.8]{PhD:Caldararu}}]\label{lemma.CaldararuindipendenzarappresentanteBrauer}
Let $\alpha\in\Br(X)$ be a cohomological Brauer class and let $(\tilde{\alpha}_1,\mU)$ and $(\tilde{\alpha}_2,\mU')$ be two representatives. Then there is an equivalence of categories
\[ \operatorname{Coh}(X,\tilde{\alpha}_1,\mU)\cong\operatorname{Coh}(X,\tilde{\alpha}_2,\mU'). \]
\end{lemma}

In view of Lemma~\ref{lemma.CaldararuindipendenzarappresentanteBrauer}, it is customary to simply talk about $\alpha$-twisted sheaves. Nevertheless, we shall keep track of the chosen cover in our notation, whenever such choice is important for some practical reasons. Notice that if $\alpha$ is trivial, then a $\alpha$-twisted sheaf is a classical sheaf and $\operatorname{Coh}(X,\alpha)\cong\operatorname{Coh}(X)$.

In a similar way, we will consider the derived category 
\[ \operatorname{D}^b(X,\alpha):=\operatorname{D}^b(\operatorname{Coh}(X,\alpha)) \] 
of $\alpha$-twisted coherent sheaves on $X$.

The following properties can be found at the end of \cite[Section~1.2]{PhD:Caldararu}. Let $F$ be an $\alpha$-twisted sheaf and $G$ a $\beta$-twisted sheaf, where $\alpha,\beta\in\Br(X)$ are two cohomological Brauer classes. Then:
\begin{itemize}
    \item $F\otimes G$ is a $\alpha\beta$-twisted sheaf;
    \item $\mathcal{H}om(F,G)$ is a $\alpha^{-1}\beta$-twisted sheaf    
\end{itemize}
Also, if $f\colon Y\to X$ is a morphism, then:
\begin{itemize}
    \item if $F$ is a $\alpha$-twisted sheaf on $X$, then $f^*F$ is an $f^*\alpha$-twisted sheaf on $Y$;
    \item if $G$ is a $f^*\alpha$-twisted sheaf on $Y$, then $f_*G$ is a $\alpha$-twisted sheaf on $X$;
    %\item $f_*$ is right adjoint of $f^*$;
    %\item if $f$ is an open immersion and $G$ is an $f^*\alpha$-twisted sheaf on $Y$, then $f_!G$ is an $\alpha$-twisted sheaf on $X$;
    %\item if $f$ is an open immersion, then $f_!$ is left adjoint of $f^*$.
\end{itemize}

In particular, it is important to point out that for any $\alpha$-twisted sheaf $F$ on $X$, the endomorphism sheaf $\mathcal{E}nd(F)$ is always a regular sheaf.

%%%%%%%%%%%%%%%%%%%%%%%%%%%%%%%%%%%%%%%%%%
\subsection{Infinitesimal deformations of twisted sheaves}\label{section:infinit def of twisted}

The aim of this section is to introduce the deformation functor associated to a twisted sheaf.

\textbf{Notation.} We shall adopt the following conventions.
\begin{itemize}
    \item For any $A\in \mathsf{Art}_{\KK}$, we shall talk about $\mathcal{O}_X\otimes A$-modules, instead of $\mathcal{O}_X\times_{\Spec(\KK)}\Spec(A)$-modules.
    \item For any $A\in \mathsf{Art}_{\KK}$ we consider the morphism $\iota_A \colon \Spec(\KK)\to \Spec(A)$ corresponding to the projection $A\to\nicefrac{A}{\mathfrak{m}_A}=\KK$. Given a $\mathcal{O}_X\otimes A$-module $F$ we shall write $F\otimes_A\KK$ to denote the pullback of $F$ under $\iota_A$. Similarly, given a $\mathcal{O}_X$-module $E$, with an abuse of notation we shall keep writing $E$ to denote its pushforward under $\iota_A$.
    \item For any $A\in \mathsf{Art}_{\KK}$, given a map of $\mathcal{O}_X\otimes A$-modules $f\colon E\to G$, we denote either by $\bar{f}\colon E\otimes_A\KK\to G\otimes_A\KK$ or by $f\otimes_A\KK$ its reduction modulo the maximal ideal of $A$.
    \item For any $A\in \mathsf{Art}_{\KK}$, given an $\mathcal{O}_X$-module $E$ we shall denote by $\pi_A\colon E\otimes A\to E$ the morphism of $\mathcal{O}_X\otimes A$-modules given by the natural projection. Notice that the reduction $\bar{\pi}_A$ is the identity on $E$.
\end{itemize}

\begin{definition}\label{def.twisteddeformation}
    Let $\mathcal{U}=\{U_i\}_{i\in I}$ be an open cover of $X$ and let $F=\left(\{F_{i}\},\{g_{ij}\},\{\alpha_{ijk}\}\right)$ be a $(\alpha,\mathcal{U})$-twisted sheaf. A $(\alpha,\mathcal{U})$-\emph{deformation} of $F$ over $A\in\mathsf{Art}_{\KK}$ is a triple $\left(\{F_{i,A}\},\{h_{i,A}\},\{g_{ij,A}\}\right)$ such that:
    \begin{enumerate}
        \item $F_{i,A}$ is a coherent sheaf on $U_i\times\Spec(A)$, flat over $\Spec(A)$, for every $i\in I$;
        \item $h_{i,A}\colon F_{i,A}\to F_i$ is a morphism of $\mathcal{O}_{U_i}\otimes A$-modules such that the reduction $\bar{h}_i\colon F_{i,A}\otimes_A \KK\to F_i$ is an isomorphism, for every $i\in I$;
        \item $g_{ij,A}\colon F_{i,A}\vert_{U_{ij}}\to F_{j,A}\vert_{U_{ij}}$ is an isomorphism of $\mathcal{O}_{U_{ij}}\otimes A$-modules such that
        \[ \xymatrix{ F_{i,A}\vert_{U_{ij}}\ar@{->}[r]^{g_{ij,A}}\ar@{->}[d]_{h_{i,A}} & F_{j,A}\vert_{U_{ij}}\ar@{->}[d]^{h_{j,A}}\\
        F_{i}\vert_{U_{ij}}\ar@{->}[r]_{g_{ij}} & F_{j}\vert_{U_{ij}}  } \]
        is a commutative diagram of $\mathcal{O}_{U_{ij}}\otimes A$-modules, for every $i,j\in I$;
        \item for every $i,j,k\in I$, the twisted cocycle condition is satisfied:
        \[ g_{ij,A}\vert_{U_{ijk}\otimes A}\circ g_{jk,A}\vert_{U_{ijk}\otimes A}\circ g_{ki,A}\vert_{U_{ijk}\otimes A} = (\alpha_{ijk}\otimes1_A)\id \; . \]
    \end{enumerate}
\end{definition}

\begin{remark}
    In the setting of Definition~\ref{def.twisteddeformation}, choosing $\{\alpha_{ijk}=1\}$ we recover the usual definition of deformation of the coherent sheaf $F$ over $A\in \Art_{\KK}$.
\end{remark}

\begin{definition}
    Let $\mathcal{U}=\{U_i\}_{i\in I}$ be an open cover of $X$ and let $F=\left(\{F_{i}\},\{g_{ij}\},\{\alpha_{ijk}\}\right)$ be an $(\alpha,\mathcal{U})$-twisted sheaf. Let $A\in\Art_{\KK}$; an \emph{isomorphism} between two $(\alpha,\mathcal{U})$-deformations $\left(\{F_{i,A}\},\{h_{i,A}\},\{g_{ij,A}\}\right)$ and $\left(\{F_{i,A}'\},\{h_{i,A}'\},\{g_{ij,A}'\}\right)$ of $F$ over $A$ is a collection of isomorphisms $\varphi_{i,A}\colon F_{i,A}\to F_{i,A}'$ such that
    \[ \xymatrix{ F_{i,A}\vert_{U_{ij}}\ar@{->}[r]^{g_{ij,A}}\ar@{->}[d]_{\varphi_{i,A}\vert_{U_{ij}}} & F_{j,A}\vert_{U_{ij}}\ar@{->}[d]^{\varphi_{j,A}\vert_{U_{ij}}}\\
    F_{i,A}'\vert_{U_{ij}}\ar@{->}[r]_{g_{ij,A}'} & F_{j,A}'\vert_{U_{ij}}  } \qquad \qquad
    \xymatrix{ F_{i,A}\ar@{->}[r]^{h_{i,A}}\ar@{->}[d]_{\varphi_{i,A}} & F_{i}\ar@{->}[d]^{\id_{F_i}}\\
    F_{i,A}'\ar@{->}[r]_{h_{i,A}'} & F_{i}  } \]
    are commutative diagrams for every $i,j\in I$.
\end{definition}

\begin{definition}\label{def.deformationfunctorconricoprimento}
    Let $F$ be an $(\alpha,\mathcal{U})$-twisted sheaf on $X$. Consider the functor
    \[ \Def^{(\alpha,\mU)}_F\colon\Art_{\KK}\longrightarrow\Sets \]
    defined on $A\in\mathsf{Art}_{\KK}$ as the set of isomorphism classes of $(\alpha,\mathcal{U})$-deformations of $F$ over $A$.
\end{definition}

\begin{remark}
    The functor $\Def^{(\alpha,\mU)}_F$ is easily described on morphisms. If $f\colon A\to B$ is a morphism in $\Art_{\KK}$ and $F_A=\left(\{F_{i,A}\},\{h_{i,A}\},\{g_{ij,A}\}\right)$ is a $(\alpha,\mathcal{U})$-deformation of $F$ over $A$, then 
    \[ \Def^{(\alpha,\mU)}_F(f)(F_A)=\left( \left(\id_{U_i}\times f^{\sharp}\right)^*F_{i,A}, \left(\id_{U_i}\times f^{\sharp}\right)^*h_{i,A}, \left(\id_{U_{ij}}\times f^{\sharp}\right)^*g_{ij,A} \right), \]
    where $f^{\sharp}\colon\Spec(B)\to\Spec(A)$.
\end{remark}

It is now important to point out that this definition does not depend neither on the choice of the representative of the class $\alpha\in\Br(X)$ nor on the choice of the open cover $\mathcal{U}$.

\begin{lemma}\label{lemma.Defnondipendedalcociclo}
    The definition of $\,\Def_F^{(\alpha,\mathcal{U})}$ does not depend on the choice of the representative $\tilde{\alpha}$ of the class $\alpha$.
    \begin{proof}
    This follows at once from the fact that the definition of twisted sheaf does not depend on the choice of the representative of the Brauer class $\alpha$, see \cite[Lemma~1.2.8]{PhD:Caldararu}. 
    \end{proof}
\end{lemma}

\begin{proposition}\label{prop:non dipende da U}
    Let $\alpha\in\Br(X)$ and $F$ a coherent $\alpha$-twisted sheaf. If $\,\mathcal{U}$ and $\mathcal{U}'$ are two open covers of $X$, then 
    \[ \Def_F^{(\alpha,\mathcal{U})}\cong\Def_F^{(\alpha,\mathcal{U}')}. \]
    \begin{proof}
        Let $\mathcal{U}=\{U_i\}_{i\in I}$ and $\mathcal{U}'=\{U_j'\}_{j\in J}$; define $\mathcal{U}''=\{U_{(i,j)}''\}_{(i,j)\in I\times J}$ where $U_{(i,j)}=U_i\cap U_j'$. For the choice of a representative $\tilde{\alpha}=\{\alpha_{ijk}\}\in \check{C}^2(\mathcal{U},\cO_{X}^*)$ of the class $\alpha\in\Br(X)$, we have a natural transformations
        \[ \Def_F^{(\alpha,\mathcal{U})}\to \Def_F^{(\alpha,\mathcal{U}'')} \]
        provided by restriction. 
        
        Notice that for any $i\in I$ the open subset $U_i$ is covered by the $U_{(i,j)}$'s, with $j$ varying in $J$. Therefore the map is an isomorphism, being $F$ untwisted on each $U_i$.
        The statement follows from Lemma~\ref{lemma.Defnondipendedalcociclo}.
        
        Putting all together we get $\Def_F^{(\alpha,\mathcal{U})}\cong \Def_F^{(\alpha,\mathcal{U}'')}\cong \Def_F^{(\alpha,\mathcal{U}')}$.
    \end{proof}
\end{proposition}

%%%%%%%%%%%%%%%%%%%%%%%%%%%%%%%%%%%%%%%%%%%%%%%%

\subsection{A global approach to deformations of twisted sheaves}\label{section:global def of infinit def of twisted}

In the previous section we chose to define deformation functors with respect to an open cover. We can refer to this as a \emph{local definition}. On the other hand there is another natural definition that is independent of the choice of an open cover, so that being a \emph{global definition}. Of course the two notions are isomorphic, and the purpose of this section is to state rigorously this identification.

\begin{definition}
    Let $X$ and $B$ be two schemes over a field $\KK$; we denote by $p_X\colon X\times B \to X$ and $p_B\colon X\times B\to B$ the two projections.
    Let $\alpha\in\Br(X)$ be a Brauer class.
    \begin{enumerate}
        \item Let $B$ be a scheme over $\KK$.
        A family of $\alpha$-twisted sheaves on $X$ parametrised by $B$ is a $p_X^{\ast}\alpha$-twisted sheaf $\mathcal{F}$ of $X\times B$ that is flat over $B$. We denote by $(\mathcal{F},B)$ a family of $\alpha$-twisted sheaves. 
        
        %\item Two families $(\mathcal{E}_1,B_1)$ and $(\mathcal{E}_2,B_2)$ of $\alpha$-twisted sheaves are \emph{isomorphic} if $B_1=B_2=B$ and there exists a line bundle $L$ on $B$ such that $\mathcal{E}_2=\mathcal{E}_1\otimes_{\cO_X}p_B^*L$.
    
        \item Let $F$ be an $\alpha$-twisted sheaf on $X$ and $(B,0)$ a pointed space. A \emph{family of deformations} of $F$ over $(B,0)$ is a pair $(\mathcal{F}_B,\phi_B)$, where $(\mathcal{F}_B,B)$ is a family of $\alpha$-twisted sheaves and $\phi_B\colon\mathcal{F}_B\otimes \KK(0)\to F$ is an isomorphism. 
        
        \item Two families of deformations $(\mathcal{F}_1,\phi_1)$ and $(\mathcal{F}_2,\phi_2)$ of $F$ over $(B,0)$ are isomorphic if there exists an isomorphism $\Phi\colon\mathcal{F}_1\to\mathcal{F}_2$ such that $\phi_2\circ(\Phi\times \KK(0))=\phi_1$.
    \end{enumerate}
\end{definition}

\begin{definition}
    For $i=1,2$, let $(B_i,0_i)$ be a pointed space and $(\cF_i,\phi_i)$ a family of deformations of $F$ over $(B_i,0_i)$. A morphism $(f,\tilde{f})\colon(B_1,\cF_1,\phi_1)\to(B_2,\cF_2,\phi_2)$ is the datum of:
    \begin{itemize}
        \item a morphism $f\colon B_1\to B_2$ such that $f(0_1)=0_2$;
        \item a morphism $\tilde{f}\colon (\id\times f)^*\cF_2\to\cF_1$ of $p_X^*\alpha$-twisted sheaves commuting with $\phi_1$ and $\phi_2$.
    \end{itemize}
\end{definition}

\begin{remark}
    If $(B_1,0_1)$ and $(B_2,0_2)$ are two pointed spaces, we denote by $p_i\colon X\times B_i\to X$ the projection on the first factor, and by $p_{B_i}\colon X\times B_i\to B_i$ the projection on the second factor. Then, in the definition above, $(\id\times f)^*\cF_2$ is a $(\id\times f)^*p_2^*\alpha=p_1^*\alpha$-twisted sheaf. This justifies the last requirement that $\tilde{f}$ is an isomorphism of $p_X^*\alpha$-twisted sheaves.
\end{remark}

\begin{definition}
    Let $F$ be an $\alpha$-twisted sheaf on $X$. Consider the functor
    \begin{equation}
        \Def^\alpha_F\colon\Art_k\longrightarrow\Sets
    \end{equation}
    defined on any $A\in\Art_k$ as the set of isomorphism classes of families of $\alpha$-deformations $(\mathcal{F}_A,\phi_A)$ of $F$ over $(\Spec(A),0)$. 
\end{definition}

\begin{remark}
    As before, the functor $\Def^\alpha_F$ is easily described on morphisms. If $f\in\operatorname{Hom}_{\Art_{\KK}}(A,B)$, then
    \[ \Def_F^\alpha(f)(\mathcal{F}_A,\phi_A)=\left(\left(\id_X\times f^\sharp\right)^*\mathcal{F}_A, \left(\id_X\times f^\sharp\right)^*\phi_A\right). \]
\end{remark}

Our aim is now to prove that the functor $\Def^\alpha_F$ is isomophic to the functor $\Def_F^{(\alpha,\mathcal{U})}$ introduced in Definition~\ref{def.deformationfunctorconricoprimento}.

\begin{proposition}\label{prop:Def alpha = Def alpha U}
    Let $F$ be an $\alpha$-twisted sheaf on $X$. For any open cover $\mU$ representing $\alpha$, there exists a natural isomorphism
    \begin{equation}\label{eqn:specifying the cover} 
        \Def^\alpha_F \longrightarrow\,\Def^{(\alpha,\mU)}_F. 
    \end{equation}
    \begin{proof}
        Let us write $F=\{F_i,\phi_{i,j}\}$ relative to an open cover $\mU=\{U_i\}$ where the twist $\alpha$ is represented. Let $(\cF_B,\phi_B)$ be a family of deformations of $F$ parametrised by $(B,0)$. Recall that $\cF_B$ is a sheaf on the product space $X\times B$ flat over $B$. Then the restrictions $\cF_{B,i}=\cF_B|_{U_i\times B}$ are families of deformations of the sheaves $F_i$ (in the classical sense). Moreover, the twist $p_X^*\alpha$ can be represented on the open cover $\mU_B=\{U_i\times B\}$ and the $p_X^*\alpha$-twisted sheaf $\cF_B$ is obtained by patching together the sheaves $\cF_{B,i}$ via the cocycle $\{\alpha_{ijk}\otimes1_B\}$. This defines the natural transformation (\ref{eqn:specifying the cover}), which is easily seen to be an isomorphism.
    \end{proof}
\end{proposition}

%%%%%%%%%%%%%%%%%%%%%%%%%%%%%%%%%%%%%%%%%

\subsection{Twisted complex vector bundles}\label{section:twisted complex vb}
For later use, we include here a quick digression on complex manifolds and complex twisted vector bundles. We refer to \cite{Perego:Twisted} for a general treatment.

Let $X$ be a complex manifold and $\alpha\in\oH^2(X,\cO_X^*)$ a twist. If $\mU=\{U_i\}$ is an open cover of $X$ where $\alpha$ is represented by the cocycle $\{\alpha_{ijk}\}$, then a \emph{$\alpha$-twisted complex vector bundle} $E$ is:
\begin{itemize}
    \item a collection of complex vector bundles $E_i$ on each $U_i$,
    \item a collection of $\mathcal{C}^{\infty}$-morphisms $g_{i,j}\colon E_{i}|_{U_{ij}}\to E_{j}|_{U_{ij}}$,
such that 
\begin{enumerate}
    \item $g_{ii}=\id$;
    \item $g_{ji}=g_{ij}^{-1}$;
    \item $g_{ij}\circ g_{jk}\circ g_{ki}=\alpha_{ijk}\id$.
\end{enumerate}
\end{itemize}

A \emph{holomorphic structure} on a $\alpha$-twisted complex vector bundle $E$ is a collection of holomorphic structures on each $E_i$, and a \emph{$\alpha$-twisted holomorphic vector bundle} is a $\alpha$-twisted vector bundle $E$ admitting a holomorphic structure such that the transinction morphisms $g_{ij}$ are holomorphic. 

In particular there is an equivalence of categories between the category of $\alpha$-twisted locally free sheaves on $X$ and the category of $\alpha$-twisted holomorphic complex vector bundles on $X$.

The definitions of deformation functors above can be carried out in this setting with minimum changes.

%%%%%%%%%%%%%%%%%%%%%%%%%%%%%%%%%%%%%%%%%

%%%%%%%%%%%%%%%%%%%%%%%%%%%%%%%%%%%%%%%%%

\section{Algebraic background}\label{section:algebra}

\subsection{DG Lie algebras and their deformation functors}\label{section:DGLA}

Recall that a DG Lie algebra is a triple $(L,d,[-,-])$ such that $L=\oplus_{n\in\mathbb{Z}}L^n$ is a graded $\KK$-vector space, $d$ is the differential and $[-,-]$ is a graded bracket satisfying some compatibility conditions (see \cite{Manetti:LMDT}).
For short, we denote by $\operatorname{DGLA}$ the category of DG Lie algebras, where morphisms are maps of DG vector spaces commuting with brackets.

We say that two DG Lie algebras $L$ and $L'$ are quasi-isomorphic if there exists a roof $L\leftarrow M \to L'$, where both maps are quasi-isomorphisms (i.e.\ the induced map on cohomology is a degreewise isomorphism).

If $L\in\operatorname{DGLA}$ is a DG Lie algebra, there is a deformation functor associated to it defined as
\[ \Def_L\colon\Art_{\KK}\to\Sets,\qquad A\mapsto\frac{\operatorname{MC}_L(A)}{\sim}, \]
where $\operatorname{MC}_L(A)=\{ \ell\in L^1\otimes\mathfrak{m}_A\mid d\ell+\frac{1}{2}[\ell,\ell]=0\}$ is the set of Maurer--Cartan solutions, and $\sim$ is the equivalence relation induced by the gauge action of $L^0\otimes\mathfrak{m}_A$ on $\operatorname{MC}_L(A)$.

%\claudio{In realtà avevamo già parlato di funtore di deformazione associato a una DGLA quando abbiamo parlato della TW totalization...forse dovremmo ripensare a questa cosa: o spostiamo questa piccola parte nella sezione precedente o la mettiamo nell'intro (comunque sono poche righe) oppure la ignoriamo del tutto.}

A given deformation functor $\mathcal{D}\colon\Art_{\KK}\to\Sets$ is said to be controlled by a DG Lie algebra $L$ if there is an isomorphism of deformation functors $\mathcal{D}\cong\Def_L$.

If $L$ and $L'$ are quasi-isomorphic DG Lie algebras, then $\Def_L\cong\Def_{L'}$ (e.g.\ \cite[Corollary~6.6.3]{Manetti:LMDT}), so that we can say that a quasi-isomorphism class of DG Lie algebras controls a given deformation problem.

Finally, let us recall two properties of DG Lie algebras that we need later. If $L\in\operatorname{DGLA}$, we say that
\begin{enumerate}
    \item $L$ is \emph{formal}, if it is quasi-isomorphic to its cohomology;
    \item $L$ is \emph{homotopy abelian}, if it is quasi-isomorphic to an abelian DG Lie algebra (i.e.\ a DG Lie algebra with trivial bracket).
\end{enumerate}
%Here we say that two DG Lie algebras $L$ and $M$ are quasi-isomorphic if there exists a roof $L\leftarrow N\to M$ of quasi-isomorphisms. 
Homotopy abelian DG Lie algebras are formal (but the vice versa is not true in general).

Geometrically, we have the following interpretations. If $L$ is homotopy abelian, then the associated deformation functor $\Def_L$ is unobstructed. If $L$ is formal, then the associated deformation functor $\Def_L$ has only quadratic obstructions, meaning that the natural morphisms
\[  \Def_L\left(\nicefrac{\KK[t]}{(t^k)}\right) \to \Def_L\left(\nicefrac{\KK[t]}{(t^2)}\right) \qquad \mbox{ and } \qquad \Def_L\left(\nicefrac{\KK[t]}{(t^3)}\right) \to \Def_L\left(\nicefrac{\KK[t]}{(t^2)}\right) \]
have the same image for every integer $k\geq3$.

\subsection{Semicosimplicial objects}
We begin by introducing semicosimplicial objects in a very general setting.

\begin{definition}
    For every non negative integer we denote by $[n]\defeq\{0,1,\ldots,n\}$ the set of non negative integers less than or equal to $n$, endowed with the usual order. The category $\overrightarrow{\Delta}$ is defined as follows:
    \begin{itemize}
        \item the objects are $\text{Ob}\left(\overrightarrow{\Delta}\right)=\left\{[n]\,\Bigl\vert\,n\geq0\right\}$,
        \item the morphisms $\Hom_{\,\overrightarrow{\Delta}\,}([n],[m])$ are strictly monotone maps.
    \end{itemize}
\end{definition}

For future purposes we now investigate concrete examples.

\begin{example}[Face maps]
    Let $n\geq1$ be an integer. Then the set of morphisms $\Hom_{\,\overrightarrow{\Delta}\,}([n-1],[n])$ consists of $n+1$ strictly monotone maps, called \emph{face maps}:
    \[ \delta_k\colon [n-1] \to [n] \qquad \delta_k(a)=
    \begin{cases}
        a & \text{ if } a<k\\
        a+1 & \text{ if } a\geq k
    \end{cases} \qquad 0\leq k\leq n \; . \]
\end{example}

With a little abuse of notation we do not emphasize the dependence on $n$ of the face map $\delta_k$.
For future reference we collect here some useful properties of face maps.

\begin{remark}[Properties of face maps]
    It is immediate to verify the following statements.
    \begin{itemize}
        \item Let $m,n\geq0$ be non negative integers. Every morphism $[n]\to[m]$ is the composition of a finite number of face maps, except for the identity.
        \item For every $0\leq h\leq k$ the \emph{semicosimplicial identities} hold:
        \begin{equation}\label{eq.semicosimplicialidentities}
        \delta_h\delta_k=\delta_{k+1}\delta_h \; .
        \end{equation}
        \item Let $n\geq0$ and $k>0$ be two given integers. Every morphism $f\colon [n]\to[n+k]$ admits a unique factorisation as composition of $k$ face maps
        \[ f=\delta_{i_1}\circ\ldots\delta_{i_k} \qquad \text{ such that } \qquad n+k\geq i_1>i_2>\ldots>i_k\geq0 \; . \]
    \end{itemize}
\end{remark}

\begin{definition}
    Let $\mathcal{C}$ be a category. A \emph{semicosimplicial object in} $\mathcal{C}$ is a covariant functor $\alpha\colon \overrightarrow{\Delta}\to\mathcal{C}$. A morphism between semicosimplicial objects in $\mathcal{C}$ is a natural transformation of functors.
\end{definition}

It is often convenient to visualise a semicosimplicial object $\alpha\colon \overrightarrow{\Delta}\to\mathcal{C}$ as a diagram
\[ \xymatrix{ \alpha_0 \ar@<-.5ex>[r] \ar@<.5ex>[r] & \alpha_1 \ar@<-.9ex>[r] \ar[r] \ar@<.9ex>[r] & \alpha_2 \ar@<-.3ex>[r] \ar@<.3ex>[r]\ar@<-1ex>[r] \ar@<1ex>[r] & \ldots } \]
where each $\alpha_n$ is an object in $\mathcal{C}$ (namely, the image of $[n]$ via the functor $\alpha$), while the morphisms between them are the images of the face maps via the functor $\alpha$.

In the same spirit, one can think of a morphism $f\colon \alpha\to \beta$ between semicosimplicial objects in $\mathcal{C}$ as a collection of maps
\[ f_n\colon \alpha_n\to \beta_n \qquad \text{ such that } \qquad \delta_kf_{n-1}=f_n\delta_k \text{ for every } n\geq1 \text{ and every } k\leq n \; . \]

We point out that some additional useful construction can be performed when the category $\mathcal{C}$ is the category of abelian groups.

\begin{definition}[Cochain complex of a semicosimplicial abelian group]\label{def.associatedcochaincomplex}
    Let $\mathcal{C}$ be the category of abelian groups and consider a semicosimplicial object $\alpha\colon \overrightarrow{\Delta}\to\mathcal{C}$. The \emph{associated cochain complex} $\left(C(\alpha),\delta\right)$ is defined as follows:
    \[ C(\alpha)^m=\begin{cases}
        0 & \text{ if } m<0\\
        \alpha_k & \text{ if } m\geq0
    \end{cases} \qquad  \delta^{n-1}\defeq\sum\limits_{k=0}^n(-1)^k\delta_k\colon C(\alpha)^{n-1}=\alpha_{n-1}\to C(\alpha)^n=\alpha_n  \text{ for every } n\geq1 \; . \]
\end{definition}

The cochain complex introduced in Definition~\ref{def.associatedcochaincomplex} is indeed a cochain complex, the condition $\delta^2=0$ being satisfied thanks to the semicosimplicial identities~\eqref{eq.semicosimplicialidentities}.

The following example will be crucial for us. Essentially the \v{C}ech cochain complex of a sheaf of abelian groups with respect to a given open covering can be tought as the cochain complex associated to a semicosimplicial object.

\begin{example}[\v{C}ech cochains]\label{example:semicosimplicial Cech}
    Let $\mathcal{F}$ be a sheaf of abelian groups on a topological space $X$. Given an open covering $\mathcal{U}=\{U_i\}_{i\in I}$ we consider the semicosimplicial abelian group of \v{C}ech cochains:
    \[ \mathcal{F}(\mathcal{U}) \colon \qquad
    \xymatrix{ \prod\limits_{i\in I} \mathcal{F}(U_i) \ar@<-.5ex>[r] \ar@<.5ex>[r] & \prod\limits_{i,j\in I} \mathcal{F}(U_{ij}) \ar@<-.9ex>[r] \ar[r] \ar@<.9ex>[r] & \prod\limits_{i,j,k\in I} \mathcal{F}(U_{ijk}) \ar@<-.3ex>[r] \ar@<.3ex>[r]\ar@<-1ex>[r] \ar@<1ex>[r] & \ldots } \; , \]
    where we denoted as usual $U_{ij}=U_i\cap U_j$ and so on. Given a face map $\delta_k\colon[n-1]\to[n]$, the induced face operator is defined as
    \[ \delta_k\colon \prod\limits_{i_0,\ldots,i_{n-1}\in I} \mathcal{F}(U_{i_0,\ldots,i_{n-1}}) \longrightarrow \prod\limits_{i_0,\ldots,i_{n}\in I} \mathcal{F}(U_{i_0,\ldots,i_n}) \qquad \qquad (\delta_k x)_{i_0,\ldots,i_n} = x_{i_0,\ldots,\hat{i_k},\ldots,i_n} \Bigl\vert_{U_{i_0,\ldots,i_n}} \; . \]
    By definition the associated cochain complex $C(\mathcal{F}(\mathcal{U}))$ is precisely the \v{C}ech complex $C(\mathcal{U},\mathcal{F})$.
\end{example}

%%%%%%%%%%%%%%%%%%%%%%%%%%%%%%%%%%%%%%%%%%%%

\subsection{Deformation functors of semicosimplicial DG Lie algebras}

\begin{definition}
    Let $\operatorname{DGLA}$ be the category of DG Lie algebras. Then a \emph{semicosimplicial DG Lie algebra} $\mathfrak{g}^\Delta$ is a semicosimplicial object in $\operatorname{DGLA}$.
\end{definition}

Our aim is now to recall how to construct a functor of Artin rings associated to a semicosimplicial DG Lie algebra $\mathfrak{g}^\Delta$. We begin by recalling the necessary background involved.

\begin{definition}[\protect{\cite[Definition~2.4]{FIM}}]
    Let $\mathfrak{g}^\Delta$ be a semicosimplicial DG Lie algebra. The functor
    \[ \mathrm{Z}^1_{\operatorname{sc}}(\exp(\mathfrak{g}^{\Delta}))\colon \mathsf{Art}_{\mathbb{K}}\to \mathsf{Set} \]
    is defined, for $A\in\mathsf{Art}_{\mathbb{K}}$, as follows
    \[ \mathrm{Z}^1_{\operatorname{sc}}(\exp(\mathfrak{g}^{\Delta}))(A) = \left\{ (l,m)\in(\mathfrak{g}_0^1\oplus\mathfrak{g}_1^0)\otimes \mathfrak{m}_A \; \left\vert \begin{array}{l}
         dl+\frac{1}{2}[l,l]=0 \\
         \partial_{1,1}l=e^m\ast\partial_{0,1}l\\
         \partial_{0,2}m\bullet-\partial_{1,2}m\bullet\partial_{2,2}m = dn+[\partial_{2,2}\partial_{0,1}l,n]\\
         \qquad \qquad \text{for some } n\in\mathfrak{g}_2^{-1}\otimes\mathfrak{m}_A
    \end{array}\right. \right\} \]
    where $\bullet$ denotes the Baker-Campbell-Hausdorff product.
\end{definition}

Fix $A\in\mathsf{Art}_{\mathbb{K}}$, then there exists an equivalence relation $\sim$ on the set $\mathrm{Z}^1_{\operatorname{sc}}(\exp(\mathfrak{g}^{\Delta})(A)$: $(l_0,m_0)\sim(l_1,m_1)$ if and only if there exist elements $a\in \mathfrak{g}_0^0\otimes \mathfrak{m}_A$ and $b\in\mathfrak{g}_1^{-1}\otimes\mathfrak{m}_A$ such that
\[ \begin{cases}
    e^a\ast l_0=l_1\\
    -m_0\bullet-\partial_{1,1}a\bullet m_1\bullet \partial_{0,1}a = db+[\partial_{0,1}l_0,b] \; .
\end{cases}\]

This is indeed an equivalence relation by~\cite[Remark 2.3]{FIM}.

\begin{definition}[\protect{\cite[Definition~2.4]{FIM}}]
    Let $\mathfrak{g}^\Delta$ be a semicosimplicial DG Lie algebra. The functor
    \[ \oH^1_{\operatorname{sc}}(\exp(\mathfrak{g}^{\Delta}))\colon \mathsf{Art}_{\mathbb{K}}\to \mathsf{Set} \]
    is defined, for $A\in\mathsf{Art}_{\mathbb{K}}$, as $\oH^1_{\operatorname{sc}}(\exp(\mathfrak{g}^{\Delta}))(A) =  \frac{\mathrm{Z}^1_{\operatorname{sc}}(\exp\mathfrak{g}^{\Delta})(A)}{\sim}$.
\end{definition}

The situation we are interested in is the following.

\begin{definition}\label{def.CechscDGLie}
    Let $\mathcal{L}$ be a sheaf of DG Lie algebras on a topological space $X$. For an open cover $\mathcal{U}=\{U_i\}$ of $X$, the \emph{\v{C}ech semicosimplicial DG Lie algebra} is
    \begin{equation}\label{eqn:Cech con L}
     \mathcal{L}^\Delta(\mathcal{U}) \colon \qquad
    \xymatrix{ \prod\limits_{i\in I} \mathcal{L}(U_i) \ar@<-.5ex>[r] \ar@<.5ex>[r] & \prod\limits_{i,j\in I} \mathcal{L}(U_{ij}) \ar@<-.9ex>[r] \ar[r] \ar@<.9ex>[r] & \prod\limits_{i,j,k\in I} \mathcal{L}(U_{ijk}) \ar@<-.3ex>[r] \ar@<.3ex>[r]\ar@<-1ex>[r] \ar@<1ex>[r] & \ldots } \; , 
    \end{equation}
\end{definition}

Therefore there is an associated functor of Artin rings $\oH^1_{\operatorname{sc}}(\exp\mathcal{L}^\Delta(\mathcal{U}))$, which depends on the chosen open cover.

\begin{example}\label{example.localmenteliberosemicosimplicial}
    Let $\mathcal{L}=\mathcal{E}nd(E)$, where $E=(E_i,g_{i,j})$ is an $(\alpha,\mathcal{U})$-twisted locally free sheaf. Let us describe the functor $\oH^1_{\operatorname{sc}}(\exp\mathcal{L}^\Delta(\mathcal{U}))$ relative to an affine open cover $\mathcal{U}$. In this case, since $\mathcal{L}$ is a sheaf of Lie algebras, we get
    \[ \mathrm{Z}^1_{\operatorname{sc}}(\exp(\mathcal{L}^{\Delta}(\mathcal{U}))(A) = \left\{\{m_{ij}\}\in\prod_{i,j\in I}\cE nd_{\cO_X}(E)(U_{ij}) \otimes\mathfrak{m}_A\,\Bigl\vert\, e^{m_{bc}\vert_{U_{abc}}}e^{-m_{ac}\vert_{U_{abc}}}e^{m_{ab}\vert_{U_{abc}}}=\id \right\}\, . \]
    Notice that $\cE nd_{\cO_X}(E)$ is defined by gluing the sheaves $\cE nd_{\cO_{U_i}}(E_i)$. In particular, the set $\cE nd_{\cO_X}(E)(U_{ij})$ is naturally identified with the sets
    \[ \cE nd_{\cO_{U_i}}(E_i)(U_{ij})\; , \quad \cE nd_{\cO_{U_j}}(E_j)(U_{ij}) \quad \text{ and } \quad \mathcal{H}om_{\cO_{U_{ij}}}(E_i\vert_{U_{ij}},E_j\vert_{U_{ij}})(U_{ij}) \]
    by suitable compositions with $g_{ij}$. Notice that in the untwisted case these are exactly the same set being $E_i\vert_{U_{ij}}=E\vert_{U_{ij}}=E_j\vert_{U_{ij}}$.
    Be careful that the cocycle condition heavily depends on this choice in the following way:
    \begin{equation*} 
         \resizebox{\textwidth}{!}{$\begin{aligned}
        \mathrm{Z}^1_{\operatorname{sc}}(\exp(\mathcal{L}^{\Delta}(\mathcal{U}))(A) &= \left\{\{m_{ij}\}\in\prod_{i,j\in I}\cE nd_{\cO_{X}}(E)(U_{ij}) \otimes\mathfrak{m}_A\,\Bigl\vert\, e^{m_{bc}\vert_{U_{abc}}}e^{-m_{ac}\vert_{U_{abc}}}e^{m_{ab}\vert_{U_{abc}}}=\id \right\}\\
        &= \left\{\{m_{ij}\}\in\prod_{i,j\in I}\cH om_{\cO_{U_{ij}}}\left(E_i\vert_{U_{ij}},E_j\vert_{U_{ij}}\right)(U_{ij}) \otimes\mathfrak{m}_A\,\Bigl\vert\, e^{m_{bc}\vert_{U_{abc}}}e^{-m_{ac}\vert_{U_{abc}}}e^{m_{ab}\vert_{U_{abc}}}=\alpha_{abc}\id \right\} \; .
    \end{aligned}$}
    \end{equation*}     
    In this paper, we shall always tacitly assume this convention. 
    %\[ m_{ij}\in\mathcal{H}om_{\cO_{U_{ij}}}(E_i\vert_{U_{ij}},E_j\vert_{U_{ij}})(U_{ij})\otimes\mathfrak{m}_A. \]
    
    Moreover, two elements $\{m_{ij}\},\{m_{ij}'\}\in \mathrm{Z}^1_{\operatorname{sc}}(\exp(\mathcal{L}^{\Delta}(\mathcal{U}))(A)$ define the same class in $\oH^1_{\operatorname{sc}}(\exp\mathcal{L}^\Delta(\mathcal{U}))(A)$ if and only if there exists an element
    \[ \{a_i\}\in \prod_{i\in I}\mathcal{E}nd_{\cO_X}(E)(U_i)\otimes\mathfrak{m}_A = \prod_{i\in I}\mathcal{E}nd_{\cO_{U_i}}(E_i)(U_i)\otimes\mathfrak{m}_A \]
    such that $e^{-a_i\vert_{U_{ij}}}\circ e^{m_{ij}}\circ e^{a_j\vert_{U_{ij}}}=e^{m_{ij}'}$ for every $i,j\in I$. %Notice that, using the identification of $\cE nd(E_i)(U_{ij})$ and $\mathcal{H}om(E_i\vert_{U_{ij}},E_j\vert_{U_{ij}})(U_{ij})$ as before, here we mean that there exists a commutative diagram
    %\[ \xymatrix{ E_i\vert_{U_{ij}}\otimes \mathfrak{m}_A \ar@{->}[r]^{m_{ij}\otimes\id}\ar@{->}[d]_-{a_i\vert_{U_{ij}}} & E_i\vert_{U_{ij}}\otimes \mathfrak{m}_A \ar@{->}[r]^{g_{ij}\otimes\id}\ar@{->}[d]^-{a_i\vert_{U_{ij}}} & E_j\vert_{U_{ij}}\otimes \mathfrak{m}_A \ar@{->}[d]^-{a_j\vert_{U_{ij}}}\\
    %E_i\vert_{U_{ij}}\otimes \mathfrak{m}_A \ar@{->}[r]_{m_{ij}\otimes\id} & E_i\vert_{U_{ij}}\otimes \mathfrak{m}_A \ar@{->}[r]_{g_{ij}\otimes\id} & E_j\vert_{U_{ij}}\otimes \mathfrak{m}_A \; . }  \]
    %are again tacitly using the identification of $\cE nd(E_i)(U_{ij})$ and $\mathcal{H}om(E_i,E_j)(U_{ij})$.
\end{example}

\begin{remark}%[\protect{\cite[Lemma~2.7]{FIM}}]
    By \cite[Lemma~2.7]{FIM}, if the open cover $\mathcal{U}$ is acyclic (with respect to $\mathcal{L}$), then there is an isomorphism of functors of Artin rings
    \[ \oH^1_{\operatorname{sc}}(\exp\mathcal{L}^\Delta(\mathcal{U}))=\varinjlim_{\mathcal{V}}\oH^1_{\operatorname{sc}}(\exp\mathcal{L}^\Delta(\mathcal{V})) \; . \]
\end{remark}

%%%%%%%%%%%%%%%%%%%%%%%%%%%%%%%%%%%%%%%%%%
\subsection{The Thom--Whitney totalization}\label{section:TW}

Associated to a semicosimplicial DG Lie algebra $\mathfrak{g}^\Delta$ there is an important DG Lie algebra, denoted by $\operatorname{Tot}_{TW}(\mathfrak{g}^{\Delta})$ and called \emph{Thom--Whitney totalization}.

We will not need the rigorous definition of $\operatorname{Tot}_{TW}(\mathfrak{g}^{\Delta})$, but only of its cohomology groups. Anyway, for completeness, we briefly recall its construction. Given a semicosimplicial DG vector space
\[ V^{\Delta} \colon \qquad \xymatrix{ V_0 \ar@<-.5ex>[r] \ar@<.5ex>[r] & V_1 \ar@<-.9ex>[r] \ar[r] \ar@<.9ex>[r] & V_2 \ar@<-.3ex>[r] \ar@<.3ex>[r]\ar@<-1ex>[r] \ar@<1ex>[r] & \ldots } \; , \]
define its Thom-Whitney totalization as
\[ \operatorname{Tot}_{TW}(V^{\Delta}) = \left\{ (x_n)\in\prod_{n\geq0}\Omega_n\otimes V_n\,\Bigl\vert\, (\delta_k^{\ast}\otimes \id)x_n = (\id\otimes\delta_k)x_{n-1} \text{ for every } 0\leq k\leq n \right\} \; , \]
where we denoted by
\[ \Omega_n = \bigoplus_{p=0}^n\Omega^p_n = \frac{\mathbb{K}[t_0,\dots,t_n,dt_0,\dots,dt_n]}{\left(1-\sum t_i,\sum dt_i\right)} \]
the DG algebra of polynomial differential forms on the affine standard $n$-symplex. Of course, the differential on $\operatorname{Tot}_{TW}(V^{\Delta})$ is induced by the one of $\prod_{n\geq0}\Omega_n\otimes V_n$. Moreover, if we endow $V^{\Delta}$ with a Lie bracket, then $\operatorname{Tot}_{TW}(V^{\Delta})$ inherits such additional structure. In fact, this will be the case in our situation.

The importance of the Thom-Whitney totalization of a semicosimplicial DG Lie algebra relies on the following result.

\begin{theorem}[\protect{\cite[Theorem~7.6]{FIM}}]\label{thm:H^1 tramite TW}
    Let $\mathfrak{g}^{\Delta}$ be a semicosimplicial DG Lie algebra. Assume that $\oH^j(\mathfrak{g}_i)=0$ for every $i\geq0$ and $j<0$. Then there is a natural isomorphism of functors of Artin rings
    \[ \oH^1_{\operatorname{sc}}(\exp\mathfrak{g}^\Delta)\cong\Def_{\operatorname{Tot}_{TW}(\mathfrak{g}^\Delta)} \; . \]
\end{theorem}

Theorem~\ref{thm:H^1 tramite TW} is telling us that the functor of Artin rings $\oH^1_{\operatorname{sc}}(\exp\mathfrak{g}^\Delta)$ naturally arises as the deformation functor associated to a DG Lie algebra. Notice that it is very deep since it provides a concrete representative of the homotopy class controlling the deformation problem $\oH^1_{\operatorname{sc}}(\exp\mathfrak{g}^\Delta)$, and this is in general a very difficult task.

Once we have the DG Lie algebra controlling a deformation problem, the first goal is to compute its cohomology in order to obtain useful information concerning first order deformations and obstructions. The following result is due to Whitney and it provides a simpler way to compute the cohomology of $\operatorname{Tot}_{TW}(\mathfrak{g}^\Delta)$.

\begin{lemma}[\protect{\cite[Theorem~7.4.5]{Manetti:LMDT}}]\label{lemma:Whitney}
    Let $V^\Delta$ be a semicosimplicial DG vector space. Then  there exists a quasi-isomorphism of DG vector spaces
    \[ \operatorname{Tot}_{TW}(V^{\Delta})\longrightarrow C(V^\Delta) \]
    where $C(V^\Delta)$ is the associated cochain complex as in Definition~\ref{def.associatedcochaincomplex}.
\end{lemma}

\begin{remark}\label{rmk.Whitney}
    When we deal with a semicosimplicial DG Lie algebra $\mathfrak{g}^\Delta$, its cochain complex $C(\mathfrak{g}^\Delta)$ is not endowed with a natural DG Lie algebra structure. This is the reason why Lemma~\ref{lemma:Whitney} is only stated for DG vector spaces, and the main motivation to introduce the Thom--Whitney totalization, which is instead a DG Lie algebra. As outlined before, by a deformation theoretic point of view, one is often mainly interested in the cohomology of the DG Lie algebra; therefore Lemma~\ref{lemma:Whitney} is particularly useful since to deal with the complex $C(\mathfrak{g}^\Delta)$ is easy with respect to the Thom--Whitney totalization.
\end{remark}

\begin{corollary}\label{corollary:H^1 is def functor}
    $\oH^1_{\operatorname{sc}}(\exp\mathfrak{g}^\Delta)$ is a deformation functor with tangent space 
    \[ T^1_{\oH^1_{\operatorname{sc}}(\exp\mathfrak{g}^\Delta)}\cong\oH^1(C(\mathfrak{g}^\Delta)) \]
    and a complete obstruction theory with values in
    \[ T^2_{\oH^1_{\operatorname{sc}}(\exp\mathfrak{g}^\Delta)}\cong\oH^2(C(\mathfrak{g}^\Delta)). \]
\end{corollary}
\begin{proof}
    Immediate from Lemma~\ref{lemma:Whitney} and Theorem~\ref{thm:H^1 tramite TW}.
\end{proof}

\begin{example}
    If $\mathfrak{g}^\Delta$ is the semicosimplicial DG Lie algebra $\mathcal{L}^\Delta(\mathcal{U})$ associated to a sheaf $\mathcal{L}$ of DG Lie algebras on $X$, and to an open cover $\mathcal{U}$ (see (\ref{eqn:Cech con L})), then 
    \[ T^1_{\oH^1_{\operatorname{sc}}(\exp\mathfrak{g}^\Delta)}\cong\oH^1(X,\mathcal{L})\quad\mbox{ and }\quad T^2_{\oH^1_{\operatorname{sc}}(\exp\mathfrak{g}^\Delta)}\cong\oH^2(X,\mathcal{L}). \]
\end{example}

\begin{example}[\protect{\cite[Theorem~6.13]{Mea}}]\label{prop:thm bello di francesco}
    Let $X$ be a Noetherian and separated scheme of finite type, and let $F$ be a quasi-coherent sheaf on $X$. Assume that there exists a locally free resolution $E^\bullet$ of $F$. If $\mathcal{L}^\Delta(\mathcal{U})$ is the semicosimplicial DG Lie algebra of $\mathcal{L}=\cE nd^\ast(E^\bullet)$, then $\operatorname{Tot}_{TW}(\mathcal{L}^\Delta(\mathcal{U}))$ is a representative of the homotopy class $\RHom_{\operatorname{D}^b(\operatorname{QCoh}(X))}(F,F)$. Notice that if we deal with a $\alpha$-twisted sheaf $F$ together with a locally free $\alpha$-twisted resolution $E^{\bullet}$, then $\cE nd^{\ast}(E^{\bullet})$ is a sheaf of DG Lie algebras and the construction of $\operatorname{Tot}_{TW}(\mathcal{L}^\Delta(\mathcal{U}))$ produces a DG Lie algebra. Since both the construction and the proof carried out in \cite{Mea} essentially rely on the local behaviour of the sheaf, we have that $\operatorname{Tot}_{TW}(\mathcal{L}^\Delta(\mathcal{U}))$ still represents the homotopy class of $\RHom_{\operatorname{D}^b(\operatorname{QCoh}(X,\alpha))}(F,F)$ as in the classical case.
\end{example}

%%%%%%%%%%%%%%%%%%%%%%%%%%%%%%%%%%%%%%%%%
%%%%%%%%%%%%%%%%%%%%%%%%%%%%%%%%%%%%%%%%%

\section{DG Lie algebra controlling deformations of twisted sheaves}\label{section:DG Lie of twisted}

The goal of this section is to determine the DG Lie algebra controlling the deformations of a twisted sheaf. The main result is the following.

\begin{theorem}\label{thm:main}
    Let $X$ be a smooth and projective variety over a field $\KK$ of characteristic $0$, or a projective complex manifold, and let $\alpha\in\Br(X)$ a Brauer class. The infinitesimal deformations of a coherent $\alpha$-twisted sheaf $F$ on $X$ are controlled by the quasi-isomorphism class of DG Lie algebras $\RHom_{(X,\alpha)}(F,F)$.
\end{theorem}

\begin{remark}
    If $\alpha=1$ is trivial, then the result is proved in \cite{FIM}.
\end{remark}

The differences in treating the algebraic and analytic situation are minimal. We will give full proofs in the algebraic setting and relegate to remarks the due changes in the analytic setting.

%\begin{remark}
%    The argument proposed here is classical and passes through the existence of locally free resolutions of $F$. This is the reason why the (quasi-)projectivity hypothesis of $X$ appears in the statement.
%\end{remark}

We divide the proof in two cases: the locally free case and the general case. The idea behind the two cases is the same, and in fact the general case is reduced to the locally free one by passing through a locally free resolution. We decided to keep the two cases separate both to improve the exposition and to make clear where the hypotheses are needed.

The following corollary can also be found in \cite{Lieblich:TwistedSheaves}.
\begin{corollary}
    Let $X$ be a smooth and projective variety over a field $\KK$ of characteristic $0$, or a projective complex manifold, and let $\alpha\in\Br(X)$ a Brauer class. If $F$ is a coherent $\alpha$-twisted sheaf on $X$ then:
    \begin{enumerate}
        \item first order deformations of $F$ are paramentrised by the space $\Ext^1_{(X,\alpha)}(F,F)$;
        \item obstructions of $F$ are contained in the space $\Ext^2_{(X,\alpha)}(F,F)$.
    \end{enumerate}
\end{corollary}

%%%%%%%%%%%%%%%%%%%%%%%%%%%%%%%%%%%%%%%%
\subsection{The locally free case}
In this case we can drop the projectivity assumptions. In the complex analytic setting, we will simply work with complex manifolds.

First of all, let us observe that any infinitesimal deformation of a locally free ($\alpha$-twisted) sheaf is locally trivial, so that the deformation data is completely determined by the gluing data.

\begin{proposition}\label{prop.deformationsaretrivialAFFINE}
    Let $\,\mathcal{U}=\{U_i\}_{i\in I}$ be an open affine cover of $X$ and let $E=\left(\{E_{i}\},\{g_{ij}\}\right)$ be a locally free $(\alpha,\mathcal{U})$-twisted sheaf.
    
    Then every deformation $\left(\{E_{i,A}\},\{h_{i,A}\},\{g_{ij,A}\}\right)$ of $E$ over $A\in\mathsf{Art}_{\KK}$ is isomorphic to a deformation of the form $\left(\{E_i\otimes A\},\{\pi_{i,A}\},\{\tilde{g}_{ij,A}\}\right)$.
    \begin{proof}
        Recall that every deformation of the locally free sheaf $E_i$ on the affine $U_i$ is trivial (see e.g.~\cite[add precise reference]{Manetti:LMDT}). Now, since $h_{i,A}\colon E_{i,A}\to E_i$ is such a deformation, for every $i\in I$ there exists an isomorphism $\varphi_{i,A}\colon E_i\otimes A \to E_{i,A}$ such that
        \[ \xymatrix{ E_i\otimes A\ar@{->}[rr]^{\varphi_{i,A}}\ar@{->}[dr]_{\pi_{i,A}} & & E_{i,A}\ar@{->}[dl]^{h_i}\\
         & E_i & } \]
         is a commutative diagram of $\mathcal{O}_X\otimes A$-modules. Now define $\tilde{g}_{ij,A}$ as the composition
         \[ \tilde{g}_{ij,A} \colon E_i\vert_{U_{ij}}\otimes A \xrightarrow{\varphi_{i,A}\vert_{U_{ij}\otimes A}} E_{i,A}\vert_{U_{ij}\otimes A} \xrightarrow{g_{ij,A}} E_{j,A}\vert_{U_{ij}\otimes A} \xrightarrow{\varphi_{j,A}^{-1}\vert_{U_{ij}\otimes A}} E_j\vert_{U_{ij}}\otimes A \; . \]
         Notice that
         \begin{itemize}
             \item $\tilde{g}_{ij,A} \circ \tilde{g}_{kj,A}^{-1} \circ \tilde{g}_{kj,A} = \varphi_i\circ\left( g_{ij,A} \circ g_{kj,A}^{-1} \circ g_{kj,A} \right)\circ \varphi_i^{-1} = (\alpha_{ijk}\otimes 1_A)\id$ for every $i,j,k\in I$,
             \item the reduction $\tilde{g}_{ij,A}\otimes_A\KK=g_{ij}$ for every $i,j\in I$.
         \end{itemize}
         Hence $\left(\{E_i\otimes A\},\{\pi_{i,A}\},\{\tilde{g}_{ij,A}\}\right)$ is a deformation of $E$ over $A$ isomorphic to $\left(\{E_{i,A}\},\{h_{i,A}\},\{g_{ij,A}\}\right)$.
    \end{proof}
\end{proposition}

\begin{remark}
    If $X$ is a complex manifold, then the same holds up to take a Stein open cover of $X$ instead. In fact, every deformation of a locally free sheaf on a Stein manifold is trivial (e.g.\ \cite[Lemma~4.2.4]{Manetti:LMDT}).
\end{remark}

\begin{theorem}\label{thm:locally free}
    Let $\,\mathcal{U}=\{U_i\}_{i\in I}$ be an open affine cover of $X$ and let $E=\left(\{E_{i}\},\{g_{ij}\}\right)$ be an $(\alpha,\mathcal{U})$-twisted sheaf. Assume that $E_i$ is locally free for every $i\in I$. There is an isomorphism of deformation functors
    \[ \Def_E^{(\alpha,\mathcal{U})}\stackrel{\cong}{\longrightarrow} \oH^1_{\operatorname{sc}}(\exp(\mathfrak{g}^{\Delta})),  \]
    where $\mathfrak{g}^{\Delta}$ is the \v{C}ech semicosimplicial Lie algebra associated to the sheaf of DG Lie algebras $\cE nd(E)$ (see Definition~\ref{def.CechscDGLie}).
    \begin{proof}
        Fix $A\in \Art_{\KK}$ and consider a deformation $\left(\{E_{i,A}\},\{h_{i,A}\},\{g_{ij,A}\}\right)$ of $E$ over $A$. By Proposition~\ref{prop.deformationsaretrivialAFFINE}, there exists an isomorphism of deformations
        \[ \varphi\colon \left(\{E_{i,A}\},\{h_{i,A}\},\{g_{ij,A}\}\right) \to \left(\{E_i\otimes A\},\{\pi_{i,A}\},\{\tilde{g}_{ij,A}\}\right) \; . \]
        where $\pi_i\colon E_i\otimes A\to E_i$ is the natural projection and 
        \[ \tilde{g}_{ij,A}=\varphi_{j,A}^{-1}\vert_{U_{ij}\otimes A} \circ g_{ij,A}\circ \varphi_{i,A}\vert_{U_{ij}\otimes A}\in\mathcal{H}om_{\cO_{U_{ij}}}\left(E_i\vert_{U_{ij}},E_j\vert_{U_{ij}}\right)(U_{ij})\otimes\mathfrak{m}_A \]
        for some isomorphisms $\varphi_i\colon E_i \otimes A\xrightarrow{\cong} E_{i,A}$ satisfying  $\pi_i = h_{i,A}\circ \varphi_i$ for every  $i\in I$.

        It is straightforward to check that the $\tilde{g}_{ij,A}$'s satisfy the cocycle identity
        \[ \tilde{g}_{ij,A}\vert_{U_{ijk}\otimes A}\circ \tilde{g}_{jk,A}\vert_{U_{ijk}\otimes A}\circ \tilde{g}_{ki,A}\vert_{U_{ijk}\otimes A}=\alpha_{ijk}\id_{U_{ijk}\otimes A} \; . \]
        It follows by Example~\ref{example.localmenteliberosemicosimplicial} that $\{\tilde{g}_{ij,A}\}\in \mathrm{Z}^1_{\operatorname{sc}}(\exp(\mathfrak{g}^{\Delta}))(A)$.

        Let us now take another choice of isomorphisms
        \[ \varphi_i'\colon E_i \otimes A\xrightarrow{\cong} E_{i,A} \qquad \text{ satisfying } \pi_i = h_{i,A}\circ \varphi'_i \qquad \text{ for every } i\in I \; . \]
        We can then define
        \[ \sigma_i = \varphi_i^{-1}\circ\varphi_i' \colon E_i\otimes A\xrightarrow{\cong} E_i\otimes A \; , \]
        so that
        \[ \begin{aligned}
            \tilde{g}_{ij,A}' = & \,(\varphi_{j,A}')^{-1}\vert_{U_{ij}\otimes A} \circ g_{ij,A}\circ \varphi_{i,A}'\vert_{U_{ij}\otimes A} = (\varphi_{j,A}\circ \sigma_j)^{-1}\vert_{U_{ij}\otimes A} \circ g_{ij,A}\circ (\varphi_{i,A}\circ\sigma_i)\vert_{U_{ij}\otimes A} \\
            = & \, (\sigma_j^{-1}\circ \varphi_{j,A}^{-1})\vert_{U_{ij}\otimes A} \circ g_{ij,A}\circ (\varphi_{i,A}\circ\sigma_i)\vert_{U_{ij}\otimes A} = \sigma_j^{-1}\vert_{U_{ij}\otimes A} \circ \tilde{g}_{ij,A}\circ \sigma_i\vert_{U_{ij}\otimes A}\, .
        \end{aligned} \]
        Again by Example~\ref{example.localmenteliberosemicosimplicial}, the above relation proves that the map
        \[ \xi_A\colon \Def_E^{(\alpha,\mathcal{U})}(A)\longrightarrow \mathrm{H}^1_{\operatorname{sc}}(\exp(\mathfrak{g}^{\Delta}))(A),\qquad \left(\{E_{i,A}\},\{h_{i,A}\},\{g_{ij,A}\}\right)\mapsto \prod_{ij}\tilde{g}_{ij,A}. \]
        is well defined.
        Moreover, by \cite[Lemma~4.2.2]{Manetti:LMDT}, the equivalence relation defined on $\mathrm{Z}^1_{\operatorname{sc}}(\exp(\mathfrak{g}^{\Delta}))(A)$ is the same as the equivalence relation of isomorphisms of the form $\sigma_i\in \End_{\cO_{U_i}}(E_i)\otimes\mathfrak{m}_A$, thus concluding that $\xi_A$ is injective. On the other hand $\xi_A$ is clearly surjective, so that it is a bijection.

        Finally, since the construction above is functorial in $A$, it defines the required isomorphism of functors 
        \[ \xi\colon \Def_E^{(\alpha,\mathcal{U})}\longrightarrow \oH^1_{\operatorname{sc}}(\exp(\mathfrak{g}^{\Delta})) \; . \]
        %This means that we can explicitly define a bijection of sets
        %\[ \Def_E^\alpha(A) \stackrel{\cong}{\longrightarrow} \oH^1_{\operatorname{sc}}(\exp(\mathfrak{g}^{\Delta}))(A) \qquad \text{ mapping } \left(\{E_{i,A}\},\{h_{i,A}\},\{g_{ij,A}\}\right) \text{ to } \{\tilde{g}_{ij,A}\} \; . \]
        %We need to check that the map above is well defined.
    \end{proof}
\end{theorem}

\begin{remark}
    The argument goes through unchanged for complex manifolds, up to choose a Stein open cover.
\end{remark}

\begin{corollary}
    Let $E$ be a $\alpha$-twisted locally free sheaf on a smooth variety $X$. Then the infinitesimal deformations of $E$ are controlled by the (homotopy class of the) DG Lie algebra $\RHom_{(X,\alpha)}(E,E)$.
    \begin{proof}
    First of all, by Proposition~\ref{prop:Def alpha = Def alpha U} it follows that for any choice of open cover $\mathcal{U}$ there is an isomorphism $\Def_E^{\alpha}\cong\Def_E^{(\alpha,\mathcal{U})}$. We can then choose an affine cover $\mathcal{U}$, so that by Theorem~\ref{thm:locally free} we have an isomorphism $\Def_E^{(\alpha,\mathcal{U})}\cong\oH^1_{\operatorname{sc}}(\exp(\mathfrak{g}^{\Delta}))$. Finally, by Theorem~\ref{thm:H^1 tramite TW} and Example~\ref{prop:thm bello di francesco} we get that $\oH^1_{\operatorname{sc}}(\exp(\mathfrak{g}^{\Delta}))$ is controlled by the (homotopy class of the) DG Lie algebra $\RHom_{(X,\alpha)}(E,E)$.
    \end{proof}
\end{corollary}

\begin{remark}\label{remark:A 0 stella di E}
    Let $X$ be a complex manifold and $E$ a $\alpha$-twisted locally free sheaf. With an abuse of notation we keep denoting by $E$ the associated complex $\alpha$-twisted vector bundle, and we denote by $\debar_E$ the holomorphic structure. This simply means that there exists an open cover $\{U_i\}$ of $X$ and $\debar_E$ is a collection of holomorphic structures $\debar_{E|_{U_i}}$ (see Section~\ref{section:twisted connections}).
    There is then a well-defined DG Lie algebra
    \[ L:=\left(A^{0,\ast}_X(\cE nd(E)),[\debar_E,-],[-,-]\right) \]
    where 
    \begin{itemize}
        \item $A^{0,\ast}_X(\cE nd(E))$ is the space of $(0,\ast)$-forms with coefficients in the vector bundle $\cE nd(E)$;
        \item $[\debar_E,-]$ is the induced holomorphic structure on $\cE nd(E)$, thus being a differential;
        \item $[-,-]$ is the Lie bracket induced by composition of sections of $\cE nd(E)$.
    \end{itemize}
    Then $L$ is a representative of $\RHom_{(X,\alpha)}(E,E)$.
\end{remark}

%%%%%%%%%%%%%%%%%%%%%%%%%%%%%%%%%%%%%%%%

\subsection{The general case}

As already remarked, the general case passes through the existence of a finite locally free resolution. Let us then start by recalling the following remark due to C\u{a}ld\u{a}raru.
\begin{lemma}[\protect{\cite[Lemma~2.1.4]{PhD:Caldararu}}]
    Let $X$ be a smooth and projective variety, $\,\mathcal{U}=\{U_i\}_{i\in I}$ an open cover of $X$ and $F=\left(\{F_{i}\},\{g_{ij}\}\right)$ an $(\alpha,\mathcal{U})$-twisted sheaf. Then there exists a finite $(\alpha,\mathcal{U})$-twisted locally free resolution of $F$, i.e.
    \[ \cdots\to E^{-1}\to E^{0}\to F \to 0. \]
    \begin{proof}
        By \cite[Lemma~2.1.4]{PhD:Caldararu} every coherent $(\alpha,\mathcal{U})$-twisted sheaf is a quotient of a locally free $(\alpha,\mathcal{U})$-twisted sheaf, from which the claim follows.
    \end{proof}
\end{lemma}

%\begin{remark}
%    The same holds for complex projective varieties. Notice that in the analytic setting we cannot relax to compact manifolds, not even for non-twisted sheaves (see \cite[Corollary~A.5]{Voisin:IMRN2002}).
%\end{remark}

\begin{theorem}\label{thm:general}
    Let $X$ be a smooth and projective variety, let $\,\mathcal{U}=\{U_i\}_{i\in I}$ be an open affine cover of $X$ and let $F=\left(\{F_{i}\},\{g_{ij}\}\right)$ be an $(\alpha,\mathcal{U})$-twisted coherent sheaf. Choose a $(\alpha,\mathcal{U})$-twisted locally free resolution $E^{\bullet}\to F$. Then there is an isomorphism of deformation functors
    \[ \Def_F^{(\alpha,\mathcal{U})}\stackrel{\cong}{\longrightarrow} \oH^1_{\operatorname{sc}}(\exp(\mathfrak{g}^{\Delta})),  \]
    where $\mathfrak{g}^{\Delta}$ is the \v{C}ech semicosimplicial Lie algebra associated to the sheaf of DG Lie algebras $\mathcal{E}nd^{\ast}(E^{\bullet})$ (see Definition~\ref{def.CechscDGLie}).
    \begin{proof}
        Let $E^{\bullet}\to F$ be an $(\alpha,\mathcal{U})$-twisted locally free resolution. In particular, $(E_i^\bullet,d_i)$ is a locally free resolution of the sheaf $F_i$ on each $U_i$, for every $i\in I$. Now consider any deformation $F_A=(\{F_{i,A}\},\{h_{i,A}\},\{g_{ij,A}\})$ of $F$ over $A\in\mathsf{Art}_{\mathbb{K}}$. Then $h_{i,A}\colon F_{i,A}\to F_i$ is a classical infinitesimal deformation over $\Spec(A)$, so that it lifts to a locally free resolution of sheaves $\left(E_i^{\bullet}\otimes A, d_i+\ell_i\right)$ over $U_i\times\Spec(A)$, for some choice of $\,\ell_i\in End^1(E_i^\bullet)\otimes \mathfrak{m}_A$, in such a way that the diagram of $\cO_{U_i}\otimes A$-modules
        \[ \xymatrix{ \left(E_i^{\bullet}\otimes A, d_i+\ell_i\right) \ar@{->}[r]^-{\epsilon_{i,A}} \ar@{->}[d]_{\pi_{i,A}} & F_{i,A}\ar@{->}[d]^{h_{i,A}}\\
        \left(E_i^{\bullet}, d_i\right) \ar@{->}[r]_-{\epsilon_{i}} & F_{i} } \]
        is commutative, for every $i\in I$. Notice that we applied Proposition~\ref{prop.deformationsaretrivialAFFINE} to state that, for every $k\in\Z_{\leq0}$ and every $i\in I$, the deformation of $E_i^k$ can be assumed of the form $E_i^{k}\otimes A\xrightarrow{\pi_{i,A}}E_i^k$.
        
        We claim that such local resolutions are glued together on double intersections  by isomorphisms of complexes $g_{ij,A}^{E^{\bullet}}$ lifting $g_{ij,A}^{F}$ and $g_{ij}^{E^{\bullet}}$; this means that we have commutative diagrams (of complexes) of $\cO_{U_{ij}}\otimes A$-modules:
        \begin{equation}\label{eq.liftings}
            \xymatrix{ \left(E_i^{\bullet}\vert_{U_{ij}}\otimes A, d_i+\ell_i\right) \ar@{->}[r]^-{\epsilon_{A,i}}\ar@{.>}[d]_{g_{ij,A}^{E^{\bullet}}} & F_{i,A}\vert_{U_{ij}}\ar@{->}[d]^-{g_{ij,A}^{F}}\\
            \left(E_j^{\bullet}\vert_{U_{ij}}\otimes A, d_j+\ell_j\right) \ar@{->}[r]_-{\epsilon_{A,j}} & F_{j,A}\vert_{U_{ij}} } \qquad \qquad 
            \xymatrix{ \left(E_i^{\bullet}\vert_{U_{ij}}\otimes A, d_i+\ell_i\right) \ar@{->}[r]\ar@{.>}[d]_{g_{ij,A}^{E^{\bullet}}} & \left(E^{\bullet}_{i}\vert_{U_{ij}},d_i\right) \ar@{->}[d]^-{g_{ij}^{E^{\bullet}}}\\
            \left(E_j^{\bullet}\vert_{U_{ij}}\otimes A, d_j+\ell_j\right) \ar@{->}[r] & \left(E^{\bullet}_{j}\vert_{U_{ij}},d_j\right) }
        \end{equation}
        To prove the claim above, consider the following morphisms of complexes
        \begin{equation*} 
         \resizebox{\textwidth}{!}{$\cH om^{\ast}_{\cO_{U_{ij}}\otimes A}\left(E_i^{\bullet}\vert_{U_{ij}}\otimes A,F_{i,A}\vert_{U_{ij}}\right) \xrightarrow[g_{ij,A}^F\circ\, -]{\cong} \cH om^{\ast}_{\cO_{U_{ij}}\otimes A}\left(E_i^{\bullet}\vert_{U_{ij}}\otimes A,F_{j,A}\vert_{U_{ij}}\right) \xleftarrow[\epsilon_{A,j}\circ\, -]{} \cH om^{\ast}_{\cO_{U_{ij}}\otimes A}\left(E_i^{\bullet}\vert_{U_{ij}}\otimes A,E_j^{\bullet}\vert_{U_{ij}}\otimes A\right) \; ; $} 
         \end{equation*}
         passing to cohomology we have isomorphisms (see e.g.~\cite[Proposition 6.2]{Mea} applied to the affine case $X=U_{ij}$). In particular, this proves that the lifting $g_{ij,A}^{E^{\bullet}}$ exists and it is uniquely defined up to cochain homotopy.
        More precisely, we actually proved that the left square in~\eqref{eq.liftings} is commutative, but we still need to check the other one. Nevertheless, taking the reduction modulo the maximal ideal $\mathfrak{m}_A$, we see that $g_{ij,A}^{E^{\bullet}}\otimes_A\KK$ lifts $g_{ij}^F$. Since also  $g_{ij}^{E^{\bullet}}$ is a lifting of $g_{ij}^F$ (and, as we saw above, liftings are unique up to homotopy), we conclude that there exists $\eta\in\cH om^{-1}_{O_{U_{ij}}}(E_i^{\bullet}\vert_{U_{ij}},E_j^{\bullet}\vert_{U_{ij}})$ such that
        \[ g_{ij,A}^{E^{\bullet}}\otimes_A\KK = g_{ij}^{E^{\bullet}} + d_j\circ\eta+\eta\circ d_i \; . \]
        Hence, if necessary, we can replace $g_{ij,A}^{E^{\bullet}}$ by
        \[ g_{ij,A}^{E^{\bullet}} + \underbrace{(d_j+\ell_j)\circ(\eta\otimes\id_{A})+(\eta\otimes\id_{A})\circ (d_i+\ell_i)}_{\text{cochain homotopy}} \]
        in order to get the required lifting.
        
        Let us now look at the cocycle condition, that reads as
        \[ g_{ij,A}\vert_{U_{ijk}\otimes A}\circ g_{jk,A}\vert_{U_{ijk}\otimes A}\circ g_{ki,A}\vert_{U_{ijk}\otimes A}=\alpha_{ijk}\id_{E_i^{\bullet}\vert_{U_{ijk}\otimes A}} + \text{homotopy} \; , \]
        since the left hand side lifts $\alpha_{ijk}\id_{F_{i,A}\vert_{U_{ijk}\otimes A}}$, and liftings are unique up to homotopy.

        Let us now recall the definition of the functor $\mathrm{Z}^1_{\operatorname{sc}}(\exp(\mathcal{L}^{\Delta}(\mathcal{U}))(A)$ in the case $\mathcal{L}=\cE nd^{\ast}(E^{\bullet})$:
        an element
        \[ (\{l_i\},\{m_{ij}\})\in\prod_{i\in I}\cE nd^{1}_{\cO_{U_i}}(E_i^{\bullet})(U_i)\oplus\prod_{i,j\in I}\cH om^{0}_{\cO_{U_{ij}}}\left(E_i^{\bullet}\vert_{U_{ij}},E_j^{\bullet}\vert_{U_{ij}}\right)(U_{ij}) \otimes\mathfrak{m}_A \]
        lies in $\mathrm{Z}^1_{\operatorname{sc}}(\exp(\mathcal{L}^{\Delta}(\mathcal{U}))(A)$ if and only if the following conditions are satisfied:
        \begin{enumerate}
            \item $d_il_i+\frac{1}{2}[l_i,l_i]=0$, that is equivalent to $(d_i+l_i)^2=0$ for every $i\in I$,
            \item $\partial_{1,1}l=e^m\ast\partial_{0,1}l$, that is equivalent to $l_j=e^{m_{ij}}\ast l_i$ for every $i,j\in I$,
            \item $\exists n\in\mathfrak{g}_2^{-1}\otimes\mathfrak{m}_A $ such that $\partial_{0,2}m\bullet-\partial_{1,2}m\bullet\partial_{2,2}m = dn+[\partial_{2,2}\partial_{0,1}l,n]$, that is equivalent to
            \[ e^{m_{bc}\vert_{U_{abc}}}e^{-m_{ac}\vert_{U_{abc}}}e^{m_{ab}\vert_{U_{abc}}}=\alpha_{abc}\id+\text{homotopy} \; . \]
        \end{enumerate}
        
        Now, for every $A\in\Art_{\KK}$, we can explicitly describe the required morphism
        \[ \begin{aligned}
            \xi\colon \Def_F^{(\alpha,\mathcal{U})}(A) &\longrightarrow \oH^1_{\operatorname{sc}}(\exp(\mathfrak{g}^{\Delta}))(A)\\
            \left(\{F_{i,A}\},\{h_{i,A}\},\left\{g^{F}_{ij,A}\right\}\right) & \longmapsto \left(\{\ell_i\},\left\{\log\left(g^{E^{\bullet}}_{ij,A}-\id\right)\right\}\right) \; ,
        \end{aligned} \]
        and it remains to be shown that $\xi$ is well defined and bijective. 
        
        By construction, the pair $\left(\{\ell_i\},\left\{\log\left(g^{E^{\bullet}}_{ij,A}-\id\right)\right\}\right)$ satisfies the conditions (1), (2), (3). Moreover, given another deformation $\left(\{F_{i,A}'\},\{h_{i,A}'\},\{g_{ij,A}'\}\right)$ of $F$ over $A$, isomorphic to $\left(\{F_{i,A}\},\{h_{i,A}\},\{g_{ij,A}\}\right)$ via the collection of isomorphisms $\{\varphi_{i,A}\colon F_{i,A}\to F_{i,A}'\}$, it is sufficient to consider the local resolutions $\epsilon_{i,A}'=\varphi_{i,A}\circ\epsilon_{i,A}$ of $F_{i,A}'$ to conclude that $\chi\left(\{F_{i,A}\},\{h_{i,A}\},\{g_{ij,A}\}\right)=\chi\left(\{F_{i,A}'\},\{h_{i,A}'\},\{g_{ij,A}'\}\right)$.
        
        We need to show that different choices of $\{\ell_i\}$ and $\{g^{E^\bullet}_{ij,A}\}$ give the same element in $\oH^1_{\operatorname{sc}}(\exp(\mathfrak{g}^{\Delta}))(A)$. To this aim, we begin by recalling the definition of the equivalence relation on $\oH^1_{\operatorname{sc}}(\exp(\mathfrak{g}^{\Delta}))(A)$ in our case: 
        $(l,m)\sim(l',m')$ if and only if there exist elements
        \[ a\in \prod\limits_{i\in I}\cE nd^{0}_{\cO_{U_i}}(E_i^{\bullet})(U_i)\otimes \mathfrak{m}_A \quad \text{ and } \quad b\in\prod\limits_{i,j\in I}\cH om^{-1}_{\cO_{U_{ij}}}\left(E_i^{\bullet}\vert_{U_{ij}},E_j^{\bullet}\vert_{U_{ij}}\right)(U_{ij}) \otimes\mathfrak{m}_A \]
        such that:
        \begin{enumerate}[(i)]
            \item $e^a\ast l=l'$, that is equivalent to $(d_i+l'_i)\circ e^{a_i} = e^{a_i}\circ (d_i+l_i)$;
            \item $-m\bullet-\partial_{1,1}a\bullet m'\bullet \partial_{0,1}a = db+[\partial_{0,1}l,b]$, that is equivalent to require that the isomorphisms of complexes
            \[ e^{m_{ij}'}\circ e^{a_i}\colon E_i^{\bullet}\vert_{U_{ij}}\otimes A\to E_j^{\bullet}\vert_{U_{ij}}\otimes A \quad \text{ and } \quad e^{a_j}\circ e^{m_{ij}}\colon E_i^{\bullet}\vert_{U_{ij}}\otimes A\to E_j^{\bullet}\vert_{U_{ij}}\otimes A \]
            are homotopic to each other for every $i,j\in I$.
        \end{enumerate}
        Now notice that in our notation we have $e^{m_{ij}}=g_{ij,A}^{E^\bullet}$, so that the morphism $\xi$ is well defined and injective. Finally, for the surjectivity of $\xi$, it is sufficient to take the $0$-th cohomology of both the complexes $\left(E_i^{\bullet}\vert_{U_{ij}}\otimes A, d_i+\ell_i\right)$ and the corresponding gluing data $g^{E^\bullet}_{ij,A}$.
    \end{proof}
\end{theorem}

\begin{remark}
    The same statement holds for projective complex manifolds, up to work with a Stein open cover.
\end{remark}

\begin{corollary}
    Let $F$ be an $\alpha$-twisted coherent sheaf on a smooth projective variety $X$, or a projective complex manifold. Then the infinitesimal deformations of $F$ are controlled by the (homotopy class of the) DG Lie algebra $\RHom_{(X,\alpha)}(F,F)$.
    \begin{proof}
        By Proposition~\ref{prop:Def alpha = Def alpha U} we can choose an affine (or Stein) open cover without changing the isomorphism class of the deformation problem. Then the claim follows at once from Theorem~\ref{thm:general}, Theorem~\ref{thm:H^1 tramite TW} and Example~\ref{prop:thm bello di francesco}.
    \end{proof}
\end{corollary}

\begin{remark}\label{rmk:A 0 stella di F}
    If $X$ is a complex manifold, following the notations introduced in Remark~\ref{remark:A 0 stella di E}, we have that $\RHom_{(X,\alpha)}(F,F)$ is represented by the DG Lie algebra
    \[ L:=\left(A^{0,\ast}_X(\cE nd^\bullet(E^\bullet)),[\debar_{E^\bullet},-],[-,-]\right). \]
\end{remark}

%%%%%%%%%%%%%%%%%%%%%%%%%%%%%%%%%%%%%%%%
%%%%%%%%%%%%%%%%%%%%%%%%%%%%%%%%%%%%%%%%

\section{Formality for twisted sheaves on minimal surfaces of Kodaira dimension 0}\label{section:surfaces}

Let $S$ be a projective minimal surface with Kodaira dimension $0$ and $F$ an $\alpha$-twisted sheaf on it. Recall that $S$ is either a K3 surface, an abelian surface, an Enriques surface or a bielliptic surface. In this section we want to prove the following result, which is a generalisation of \cite[Theorem~1.1]{BMM:Formality}.

\begin{theorem}\label{thm:kodaira 0}
Let $S$ be a smooth projective minimal surface of Kodaira dimension $0$, and let $H$ an ample line bundle on it. If $F$ an $\alpha$-twisted sheaf that is $H$-polystable, then $\RHom_{(S,\alpha)}(F,F)$ is a formal DG Lie algebra.
\end{theorem}

The notion of (semi)stability for twisted sheaves is the analogous for non-twisted ones. We refer to \cite[Section~3]{HuybrechtsStellari:ProofofCaldararu} for a survey of the several definitions and their equivalence; we also refer to \cite[Section~4.2]{Perego:Twisted} for the relevant definitions and results for slope (semi)stability of locally free twisted sheaves. 

\begin{remark}
    If $S$ is a projective K3 or abelian surface, then by Serre duality we get the stronger statement that for an $H$-stable $\alpha$-twisted sheaf $F$ the DG Lie algebra $\RHom_{(S,\alpha)}(F,F)$ is homotopy abelian.
\end{remark}

Theorem~\ref{thm:kodaira 0} will be proved in Section~\ref{section:proof of Kodaira 0}, let us now see a consequence.

\begin{corollary}\label{cor:formality on the derived category}
    Let $S$ be either a projective  K3 surface or an abelian surface, $\sigma\in\operatorname{Stab}^\dag(S)$ a full numerical geometric stability condition and $F\in\Db(S)$ a $\sigma$-polystable object. Then $\RHom_S(F,F)$ is formal, and it is homotopy abelian as soon as $F$ is $\sigma$-stable.
\end{corollary}
\begin{proof}
    In the proofs of \cite[Theorem~4.1.1]{MinimadeYanagidaYoshioka:SomeModuliSpaces} (for abelian surfaces) and \cite[Lemma~7.3]{BayerMacriProjectivity} (for K3 surfaces) it is proved that there exists another projective symplectic surface $T$, a Brauer class $\alpha\in\Br(T)$ and a derived equivalence 
    \[ \Phi\colon\operatorname{D}^b(S)\longrightarrow\operatorname{D}^b(T,\alpha) \]
    such that $\Phi(F)\in\operatorname{Coh}(T,\alpha)$. Moreover $\Phi(F)$ is Gieseker (poly)stable when $F$ is Bridgeland (poly)stable, inducing an isomorphism between the corresponding moduli spaces. 

    To conclude the proof is then enough to apply \cite[Corollary~2.15]{Budur_Zhang_2019}, which says that $\operatorname{RHom}_S(F,F)$ is formal if and only if $\operatorname{RHom}_{(S,\alpha)}(\Phi(F),\Phi(F))$ is, and deduce the claim from Theorem~\ref{thm:kodaira 0}. Notice that \cite[Corollary~2.15]{Budur_Zhang_2019} holds as soon as $\operatorname{D}^b(X,\alpha)$ and $\operatorname{D}^b(Y,\beta)$ have a strongly unique DG enhancement. The latter follows from \cite[Remark~5.15.(i)]{CANONACO201728}
\end{proof}
\begin{remark}
    Corollary~\ref{cor:formality on the derived category} for K3 surfaces already appeared in \cite[Theorem~3.2]{ASformality2}. Let us remark that in \cite{ASformality2} it is further proved the analogous statement for polystable objects in K3 categories of cubic fourfolds.
\end{remark}

%%%%%%%%%%%%%%%%%%%%%%%%%%%%%%%%%%%%%%%
\subsection{Proof of Theorem~\ref{thm:kodaira 0}}\label{section:proof of Kodaira 0}
We closely follow the proof of \cite[Theorem~1.1]{BMM:Formality}.

Since $S$ is a minimal surface of Kodaira dimension $0$, the canonical bundle $K_S$ is torsion, i.e.\ there exists $n>0$ such that $K_S^n=\cO_S$. We can then consider the locally free sheaf of commutative $\cO_S$-algebras
\[ \cC_S:=\cO_S\oplus K_S \oplus \cdots \oplus K_S^{n-1}. \]
If $E^\bullet\to F$ is a locally free resolution of $F$, then we also consider the complex of (non-twisted) $\cO_S$-modules
\[ \cE nd^\ast(E^\bullet\otimes\cC_S), \]
and define the DG Lie algebra
\[ \widetilde{L}:=A^{0,\ast}_S\left(\cE nd^\ast(E^\bullet\otimes\cC_S)\right). \]

\begin{claim*}
    Under the hypothesis of Theorem~\ref{thm:kodaira 0}, the DG Lie algebra $\widetilde{L}$ is formal.
\end{claim*}

The claim allows us to conclude the proof of Theorem~\ref{thm:kodaira 0} as in the last part of the proof of \cite[Theorem~5.3]{BMM:Formality}. In fact, there is an inclusion of DG Lie algebras
\[ A^{0,\ast}_S\left(\cE nd^\ast(E^\bullet)\right)\subset A^{0,\ast}_S\left(\cE nd^\ast(E^\bullet\otimes\cC_S)\right) \]
and the formality of $A^{0,\ast}_S\left(\cE nd^\ast(E^\bullet)\right)$ follows from the formality of $A^{0,\ast}_S\left(\cE nd^\ast(E^\bullet\otimes\cC_S)\right)$ by the formality transfer of \cite{Manetti:FormalityTransfer} (see also \cite[Theorem~2.3]{BMM:Formality}). 

By Remark~\ref{rmk:A 0 stella di F}, the DG Lie algebra $A^{0,\ast}_S\left(\cE nd^\ast(E^\bullet)\right)$ is a representative of $\RHom_{(S,\alpha)}(F,F)$, thus concluding the proof.

\proof[Proof of the Claim.]
Let us collect some facts. 
The first one is well-known.

\begin{lemma}\label{lemma:polystable is reductive}
    Let $X$ be a smooth and projective variety and $H$ an ample line bundle. If $F$ is a $H$-polystable $\alpha$-twisted sheaf on $X$, then the group $\Aut_X(F)$ is linearly reductive. \qed
\end{lemma}

The second one asserts the existence of special $\alpha$-twisted locally free resolutions of a coherent $\alpha$-twisted sheaf.

\begin{lemma}\label{lemma:equivariant and finitely supported resolutions}
    Let $X$ be a smooth and projective variety and $H$ an ample line bundle. If $F$ is a $H$-polystable $\alpha$-twisted sheaf, then there exists a $\alpha$-twisted locally free resolution $E^\bullet\to F$ such that 
    \begin{enumerate}
        \item $E^\bullet\to F$ is $\Aut_X(F)$-equivariant;
        \item the action of $\Aut_X(F)$ on $E^\bullet$ is rational and finitely supported, i.e.\ for every open subset $U\subset X$ the space of sections $\Gamma(U,E^\bullet)$ is degree-wise a rational representation of every subgroup $G\leq\Aut_X(F)$ that is supported on finitely many irreducible representations.
        \item If $X$ is a minimal surface of Kodaira dimension $0$ and $\cC_X$ is the locally free sheaf of commutative $\cO_S$-algebras defined above, then every endomorphism of $F\otimes\cC_X$ lifts to an endomorphism of $E^\bullet\otimes\cC_X$.
    \end{enumerate}
    \begin{proof}
        Let us write $F=\bigoplus_{k=1}^n F_k\otimes V_k$, where $F_k$ is stable and $V_k$ is a vector space. Let us also consider the equivalence relation 
        \[ F_i\sim F_j\;\Leftrightarrow F_i\cong F_j\otimes K_X^m\;\;\mbox{for some } m. \]
        Then, if $F_1,\dots, F_r$ is a set of representatives we can write
        \[ F\cong\bigoplus_{k=1}^r F_k\otimes\mathcal{W}_k\otimes V_k, \]
        where $\mathcal{W}_k$ is a direct sum of line bundles of the form $K_X^m$ (for appropriate exponents). 
        
        Now, for every $k$, let us consider a $\alpha$-twisted locally free resolution $E_k^\bullet\to F_k$ and define 
        \[ E^\bullet:=\bigoplus_{k=1}^r E_k^\bullet\otimes\mathcal{W}_k\otimes V_k. \]
        Clearly $E^\bullet$ is a $\alpha$-twisted locally free resolution of $F$, and it is easy to see that it satisfies all the sought conditions.
    \end{proof}
\end{lemma}

The last fact is about the properties of the DG Lie algebra $\widetilde{L}$.

\begin{lemma}\label{lemma:quasi-cyclic}
    The DG Lie algebra $\widetilde{L}$ above is quasi-cyclic of degree $2$, i.e.\ there exists a symmetric bilinear form
    \[ (-,-)\colon \widetilde{L}\times \widetilde{L} \longrightarrow \C[-2] \]
    that is non-degenerate in cohomology and such that 
    \begin{equation}\label{eqn:quasi-cyclic} 
    (dx,y)=(-1)^{\bar{x}+1}\,(x,dy)\quad\mbox{ and }\quad ([x,y],z)=(x,[y,z]) 
    \end{equation}
    for every homogeneous elements $x,y,z\in\widetilde L$ (here we denote by $\bar{x}$ the degree of the homogeneous element $x\in\widetilde L$).
    \begin{proof}
        First of all, let us consider the trace map
        \[ \widetilde{\tr}\colon \cE nd_S^\ast(E^\bullet\otimes\cC_S)\to \cC_S. \]
        This map is the extension of the more natural trace map $\tr\colon \cE nd_S^\ast(E^\bullet)\to\cO_S$, which is defined locally as in the untwisted case (see e.g.\ \cite[Section~2.2.1]{Perego:Twisted}).

        Then we can define
        \[ (-,-)\colon \widetilde{L}\times\widetilde{L}\to\C[-2],\qquad (x,y)\mapsto \int_S \pi\left(\widetilde{\tr}(xy)\right), \]
        where $\pi\colon A^{0,\ast}_S(\cC_S)\to A^{0,2}(K_S)$ is the projection on the given direct summand.

        It is straightforward to see that $(-,-)$ is symmetric. Moreover, it follows from the usual properties of the trace map that the relations (\ref{eqn:quasi-cyclic}) holds. Finally, $(-,-)$ is non-degenerate in cohomology by Serre duality.
    \end{proof}
\end{lemma}

Let us conclude the proof of the claim. First of all, let us notice that since $F$ is polystable, also $F\otimes\cC_S$ is polystable. In particular, by Lemma~\ref{lemma:polystable is reductive} the group $\Aut_{(S,\alpha)}(F\otimes\cC_S)$ is linearly reductive. 

Let us now fix once and for all a resolution $E^\bullet\to F$ as in Lemma~\ref{lemma:equivariant and finitely supported resolutions}. Then $\Aut_{(S,\alpha)}(F\otimes\cC_S)$ acts on $\widetilde L$ and, arguing as in the proof of \cite[Theorem~5.1]{BMM:Formality}, this action is rational and finitely supported. 

Since $\widetilde L$ is quasi-cyclic of degree $2$ (Lemma~\ref{lemma:quasi-cyclic}), its formality follows from \cite[Corollary~4.5]{BMM:Formality}.
\endproof

%%%%%%%%%%%%%%%%%%%%%%%%%%%%%%%%%%%%%%%%%%%
\section{Formality for twisted sheaves on hyperk\"ahler manifolds}\label{section:HK}

In \cite{MO22} we proved that the DG Lie algebra controlling the deformations of a hyper-holomorphic vector bundle on a hyper-K\"ahler manifold is formal. In this section we extend that result to twisted vector bundles. 

\subsection{A quick guide to hyper-K\"ahler manifolds and hyper-holomorphic bundles}
Let us start by recalling the main definitions and notions. If $X$ is a complex manifold, we also write $X=(M,I)$ when we want to specify the underlying differential manifold $M$ and the complex structure $I$. If $g$ is a riemannian metric on $M$, then $\omega_I=g(I(-),-)$ is the associated $2$-form of type $(1,1)$; recall that $(M,g,I)$ is K\"ahler if $\omega_I$ is closed. An \emph{hyper-K\"ahler} manifold is then a $5$-uple $(M,g,I,J,K)$ where $I$, $J$ and $K$ are complex structures satisfying the quaternionic relations and such that the induced $2$-forms $\omega_I$, $\omega_J$ and $\omega_K$ are K\"ahler. If we fix the complex structure $I$, then the $2$-form $\sigma_I=\omega_J+i\omega_K$ is of type $(2,0)$ and non-degenerate, i.e.\ it is an holomorphic symplectic form on $X=(M,I)$. Vice versa, Yau's solution to Calabi's conjecture says that if $X=(M,I)$ is a compact K\"ahler manifold with an holomorphic symplectic form, then for any K\"ahler class $\omega$ there exists a riemannian metric $g$ and a hyper-K\"ahler structure $(M,g,I,J,K)$ such that $\omega=\omega_I$. 

Once a hyper-K\"ahler structure is fixed (for example, by fixing a K\"ahler class), there is a natural action of the group $\SU(2)$ on the complex tangent bundle $T_M$. This action is the one generated by the actions of $I$, $J$ and $K$. Extending multiplicatively, we can further define an action of $\SU(2)$ on the de Rham algebra $A^*_X$.

If $E$ is a complex vector bundle on a hyper-K\"ahler manifold $X$, then we can define an action of $\SU(2)$ on the algebra $A^*_X(E)$ by letting $\SU(2)$ act trivially on the coefficients, i.e.\ on the sections of $E$. 
Suppose that $E$ is a holomorphic vector bundle. We say that $E$ is $\omega$-\emph{hyper-holomorphic} if there exists an hermitian metric and a Chern connection $\nabla$ on $E$ such that the curvature $\nabla^2\in A^{1,1}(End(E))$ is $\SU(2)$-invariant, where the action of $\SU(2)$ is associated to the K\"ahler class $\omega$. If this is the case, the connection $\nabla$ is Yang--Mills.

Among the main properties of being hyper-holomorphic, we want to recall here that any such vector bundle deforms holomorphically along the twistor family. More precisely, we can endow the product $M\times S^2$ with a complex structure $\cX=(M\times S^2,\tilde{I})$ such that the projection $p\colon\cX\to\mathbb{P}^1$ is holomorphic and $p^{-1}(a,b,c)=(M,aI+bJ+cK)$. The other projection $q\colon\cX\to X$ is never holomorphic and hyper-holomorphic vector bundles $E$ on $X$ are characterised by the property that $q^*E$ is an holomorphic vector bundle on $\cX$ (\cite[Lemma~5.1]{KV}). 

Finally, let us recall that a holomorphic vector bundle $E$ is $\omega$-\emph{projectively
hyper-holomorphic} if there exists a metric connection on $E$ such that the induced connection on $End(E)$ is $\omega$-hyper-holomorphic. 

\subsection{Connections on twisted vector bundles}\label{section:twisted connections}
We refer to Section~\ref{section:twisted complex vb} for the definition of twisted complex vector bundle, and to \cite{Perego:Twisted} for a general account on connections on them.

Let $X=(M,I)$ be a complex manifold and $\alpha\in\Br'(X)$ a twist. We further fix here a $B$-field associated to the twist. In particular, if $\mathfrak{U}=\{U_i\}$ is an open cover where $\alpha$ is represented, this means that we fix a collection $\{B_i\}$ of $2$-forms on each $U_i$ and a collection $\{\eta_{ij}\}$ of $1$-forms on each $U_{ij}$ such that $B_i-B_j=d\eta_{ij}$. 
Consider now the morphism
\[ \gamma\colon\oH^2(X,\cO_X^*)\longrightarrow\oH^3(X,\Z) \]
induced by the exponential sequence.
By definition the $3$-forms $\{dB_i\}$ glue to form a global and closed $3$-form $dB$, and the twist and the $B$-field are associated if $\gamma(\alpha)=[dB]$.
Following \cite{Perego:Twisted}, we further assume that the $B_i$ are purely imaginary forms of type $(1,1)$ and the $\eta_{ij}$ are forms of type $(1,0)$.
\medskip

Let $E$ be an $\alpha$-twisted complex vector bundle on $X$. A \emph{connection} on $E$ is a collection of connections $\{\nabla_i\}$ on each $E_i$ satisfying a twisted compatibility condition. More precisely, if we represent each $\nabla_i$ with a matrix $\Gamma_i$ of $1$-forms relative to a local frame, then 
\[ \Gamma_i=\phi_{ij}^{-1}\Gamma_j\phi_{ij}+\phi_{ij}^{-1}d\phi_{ij}+\eta_{ij}\id. \]

As usual we can decompose $\nabla$ in its $(1,0)$ and $(0,1)$ parts. The assumption that $\omega_{ij}$ are of type $(1,0)$ then implies that $\nabla^{0,1}$ is a semi-connection in the classical sense. In particular, if the twisted vector bundle is holomorphic, i.e.\ there is a holomorphic structure $\debar_E=\{\debar_{E_i}\}$, then $\nabla$ is compatible with the holomorphic structure if $\nabla^{0,1}=\debar_E$.

An hermitian metric on $E$ is a collection of hermitian metrics on each $E_i$. A connection is then \emph{metric} if it is compatible with the metric (see \cite[Section~2.5]{Perego:Twisted}). If the $\alpha$-twisted vector bundle is holomorphic and we fix a hermitian metric on $E$, then there is a unique connection, called \emph{twisted Chern connection}, compatible with both the holomorphic structure and the metric (\cite[Lemma~2.25]{Perego:Twisted}).

Finally, if $\nabla$ is a connection on a $\alpha$-twisted vector bundle $E$, then we can induce a connection $\tilde{\nabla}$ on the (untwisted) vector bundle $End(E)$ as follows. Let $\mU=\{U_i\}$ be an open cover representing $\alpha$, $E$ and $\nabla$. Then $End(E)$ is obtained by gluing together the vector bundles $End(E_i)$ in the usual way. The connections $\nabla_i$ induce now connections $\tilde{\nabla}_i$ on each $End(E_i)$, which can again be glued together to form a connection $\tilde{\nabla}$ on $End(E)$ (see also \cite[Section~2.6]{Perego:Twisted}).

%%%%%%%%%%%%%%%%%%%%%%%%%%%%%%%%%%%%%%%%%%%%%%%%
\subsection{Projectively hyper-holomorphic twisted vector bundles}

The following is a very natural definition, see for example \cite{Markman:modular,MSYZ:D-Equiv,Bottini:OG10twisted}.

\begin{definition}\label{defn:twisted proj hyper-holomorphic}
    Let $X$ be a hyper-K\"ahler manifold, $\omega$ a K\"ahler class and $E$ a $\alpha$-twisted holomorphic vector bundle. Then $E$ is called $\omega$-\emph{projectively hyper-holomorphic} if there exists a hermitian metric, inducing a twisted Chern connection $\nabla$ on $E$, such that the induced connection $\tilde{\nabla}$ on $End(E)$ is $\omega$-hyper-holomorphic.
\end{definition}

\begin{remark}[Speculation on the notion of twisted hyper-holomorphic connections]\label{rmk:speculation}
    If $\nabla$ is a connection on an $\alpha$-twisted vector bundle and $\{B_i,\eta_{ij}\}$ is a $B$-field associated to $\alpha$, then the curvature $R_\nabla$ of $\nabla$ is obtained (cf.\ \cite[Definition~2.16]{Perego:Twisted}) by gluing together the local forms
    \[ R_i=\nabla_i^2-B_i\otimes\id\in A^2_X(End(E_i)). \]
    If we assume that the $2$-forms $B_i$ are $\SU(2)$-invariant, then it would be natural to define an $\omega$-\emph{hyper-holomorphic} connection as a connection $\nabla$ such that each $\nabla_i$ is $\omega$-hyper-holomorphic. This definition is then equivalent to ask that the curvature $R_\nabla$ is $\SU(2)$-invariant. Notice that the B-field being $\SU(2)$-invariant implies that $B_i\in A^{1,1}_{U_i}$ so that $R_\nabla\in A^{1,1}_X(End(E))$ in this case.

    It would be very interesting to have a good theory of twisted hyper-holomorphic connections as the one developed in \cite{Verbitsky:hyperholomorphic}. For example, using Perego's theory of twisted Yang--Mills connections (\cite{Perego:Twisted}), it is not difficult to see that such a connection is (twisted) Yang--Mills and, if metric, its $(0,1)$ part induces a holomorphic structure on $E$. Moreover, we expect that the following statement holds:
    \begin{center}
        if $E$ is a slope stable holomorphic vector bundle with twisted first and second Chern classes that are $\SU(2)$-invariants, then the unique Yang--Mills connection is hyper-holomorphic.
    \end{center}

    This statement passes through the twisted Kobayashi--Hitchin correspondence, which is now proved in \cite[Section~5]{Perego:Twisted}. The twisted Chern classes are defined by twisting by the classical Chern classes by the B-field (see \cite[Section~3.1]{Perego:Twisted} and \cite[Section~1]{HuybrechtsStellari}).

    Our definition of twisted projectively hyper-holomorphic connections by-passes these questions inasmuch as it is a condition on the induced connection on the endomorphism bundle, which is untwisted. We do not need a complete theory of twisted hyper-holomorphic connections, which will be studied elsewhere.
\end{remark} 

\begin{remark}
    As in the untwisted case, in Definition~\ref{defn:twisted proj hyper-holomorphic} we assume that the connection is metric (for example, if $End(E)$ is slope polystable, we can take the unique Yang--Mills metric). This assumption will be essential in our proof of Theorem~\ref{thm:formality of twisted hyper-holomorphic} below, as it is essential in \cite[Theorem~4.8]{MO22}.
\end{remark}

\begin{remark}
    As we have already remarked, by \cite[Lemma~5.1]{KV} the holomorphic vector bundle $End(E)$ being hyper-holomorphic means that it deforms along the twistor family. In particular this means that the projectified bundle $\mathbb{P}(E)$ deforms (as a projective bundle) along the twistor family. Twisted projectively hyper-holomorphic vector bundles are then the natural analog of (untwisted) projectively hyper-holomorphic vector bundles.
\end{remark}

\begin{theorem}\label{thm:formality of twisted hyper-holomorphic}
    Let $E$ be a $\alpha$-twisted projectively hyper-holomorphic vector bundle on a compact hyper-K\"ahler manifold $X$. Then the DG Lie algebra $A^{0,*}_X(End(E))$ is formal.
\end{theorem}
\begin{proof}
    Let $\nabla$ be a connection on $E$ and $\tilde{\nabla}$ the induced connection on $End(E)$. If $(M,g,I,J,K)$ is the hyper-K\"ahler structure on $X$, we define the operator
    \[ \debar_J:=J^{-1}\circ\tilde{\nabla}^{1,0}\circ J\colon A^{0,*}_X(End(E))\longrightarrow A^{0,*}_X(End(E)). \]
    We claim that the triple 
    \[ (A^{0,*}_X(End(E)),\debar=\tilde{\nabla}^{0,1},\debar_J) \] 
    is a DGMS algebra as defined in \cite[Definition~1.5]{MO22}. Then the theorem follows at once from \cite[Theorem~1.8]{MO22}.
    We need to prove the following statements:
    \begin{enumerate}
        \item the pair $(A^{0,*}_X(End(E)),\debar_J)$ is an associative DG algebra;
        \item $\debar_J^2=0$ and $[\debar,\debar_J]=0$;
        \item the strong $\debar\debar_J$-lemma holds.
    \end{enumerate}

    Let us show item (1). First of all, since $End(E)$ (resp.\ $\tilde{\nabla}$) is obtained by gluing together $End(E_i)$ (resp.\ $\tilde{\nabla}_i$), it follows that $\tilde{\nabla}$ acts as a derivation on $A^{0,*}_X(End(E))$. In fact, each $\tilde{\nabla}_i$ acts as a derivation on $A^{0,*}_X(End(E_i))$ by a direct computation (cf.\ \cite[Remark~3.12]{MO22}). Therefore the difference $\tilde{\nabla}^{1,0}=\tilde{\nabla}-\tilde{\nabla}^{0,1}$ acts as a derivation. Now, by definition $J$ acts multiplicatively on $A^{0,*}_X(End(E))$ so that $\debar_J$ must act as a derivation on $A^{0,*}_X(End(E))$, turning it into an associative DG algebra as claimed.

    Finally, items (2) and (3) follow from \cite[Proposition~3.7, Lemma~3.9]{MO22} since $\tilde{\nabla}$ is hyper-holomorphic by definition (and it is metric by \cite[Lemma~2.33]{Perego:Twisted}).
\end{proof}

\begin{example}
    Any slope stable $\alpha$-twisted locally free sheaf $E$ on a K3 surface is projectively hyper-holomorphic. The proof of this fact follows as in the proof of \cite[Proposition~2.3]{HuySchr:Brauer} (up to using the twisted Kobayashi--Hitchin correspondence in \cite{Perego:Twisted}).
\end{example}

\begin{example}
    Several examples of projectively hyper-holomorphic $\alpha$-twisted locally free sheaves on hyper-K\"ahler manifolds of type $\operatorname{K3}^{[2]}$ have been studied in \cite{Bottini:OG10twisted}.
\end{example}

%%%%%%%%%%%%%%%%%%%%%%%%%%%%%%%%%%%%%%%%%%%%%
%%%%%%%%%%%%% BIBLIOGRAPHY %%%%%%%%%%%%%%%%%%
%%%%%%%%%%%%%%%%%%%%%%%%%%%%%%%%%%%%%%%%%%%%%
\bibliographystyle{amsplain-nodash}
\bibliography{bib}

\end{document}